\documentclass{article}
\usepackage{amsmath,amssymb,paralist}
\usepackage{amscd}
\usepackage{latexsym,amsfonts,amsthm,amsmath,mathrsfs}
\usepackage{graphicx}
\usepackage[margin=1.25in]{geometry}
\newtheorem{theo}{Theorem}[section]

\newtheorem{remark}[theo]{Remark}
\newtheorem{assumptions}[theo]{Hypotheses}
\newtheorem{proposition}[theo]{Proposition}

 \newtheorem{definition}[theo]{Definition}
 \numberwithin{equation}{section}

\def\beq{\begin{equation}}
\def\eeq{\end{equation}}
\def\beqa{\begin{eqnarray}}
\def\eeqa{\end{eqnarray}}
\def\beqn{\begin{eqnarray*}}
\def\eeqn{\end{eqnarray*}}
\def\cF{{\mathcal F}}
\def\cFv{{\mathcal F}^\ast}
\def\cA{{\mathcal A}}
\def\R{\mathbb{R}}
\def\eps{\varepsilon}

\def\Q{Q}

\def\W{W}
\def\cI{\mathscr{I}}
\def\cH{\mathcal{H}}
\def\cL{\mathcal{L}}
\def\cI{\mathcal{I}}

\def\d{d}
\def\N{d}
\def\n{n}
\def\m{k}

\newcommand{\mres}{\mathbin{\vrule height 1.6ex depth 0pt width
0.13ex\vrule height 0.13ex depth 0pt width 1.3ex}}

\title{}
\title{Domain formation in membranes near the onset of instability}

\author{
	Irene Fonseca \\
	Carnegie Mellon University \\
	Pittsburgh, PA, USA \\
	fonseca@andrew.cmu.edu
	\and
	Gurgen Hayrapetyan\\
	Ohio University \\
	Athens, OH, USA \\
	hayrapet@ohio.edu
	\and
	Giovanni Leoni\\
	Carnegie Mellon University \\
	Pittsburgh, PA, USA \\
	giovanni@andrew.cmu.edu
	\and
	Barbara Zwicknagl\\
	Bonn University \\
	Bonn, Germany \\
	zwicknagl@iam.uni-bonn.de	
	}

\date{\today}

\begin{document}

\maketitle

\begin{abstract}
The formation of microdomains, also called rafts, in biomembranes can be attributed to the surface tension of the membrane. In order to model this phenomenon, a model involving a  coupling between the local composition and the local curvature was proposed by Seul and Andelman in 1995.
In addition to the familiar Cahn-Hilliard/Modica-Mortola energy, there are additional `forces' that prevent large domains of homogeneous concentration. This is taken into account by the bending energy of the membrane, which is coupled to the value of the order parameter, and reflects the notion that surface tension associated with a slightly curved membrane influences the localization of phases as the geometry of the lipids has an effect on the preferred placement on the membrane.

The main result of the paper is the study of the $\Gamma$-convergence of this family of energy  functionals, involving nonlocal as well as negative terms.
Since the minimizers of the limiting energy have minimal interfaces, the physical interpretation is that, within a sufficiently strong interspecies surface tension and a large enough sample size, raft microdomains are not formed.  
\end{abstract}



\section{Introduction}

The continuum theory of  membranes has been an active area of research in material and biological sciences since the pioneering works of Canham and Helfrich, \cite{Canham70, Helfrich73}. Biological cell membranes or biomembranes are complex structures commonly made up of lipids, proteins, and cholesterol.   Of recent very widespread interest is the phase separation and domain formation of these compounds forming the cell membrane.  The resulting nanoscale microdomains, referred to as `lipid rafts', are believed to be responsible for membrane trafficking, intracellular signaling, and assembly
of specialized structures, \cite{simons1997functional}. Many important biological processes, such as virus budding,
endocytosis, and immune responses, are believed to be linked to membrane rafts, \cite{rajendran2005lipid}.  Ever since the first experimental evidence of raft formation in late 1980's, there has been a growing body of literature on both theoretical and experimental aspects of this phenomenon, \cite{elson2010phase}.
However, due to very small scales associated with raft domains (they are too small to be optically resolved) \cite{rajendran2005lipid, brown1998functions, meinhardt2013monolayer},  there are different viewpoints on the precise structure and stability of lipid rafts,  \cite{leslie2011lipid}.   As a result, understanding the conditions for the formation, as well as  mechanisms driving stability (and instability), of these microdomains is of great importance.

It has been proposed in \cite{Komura:Langmuir:2006} that raft formation can be attributed to the surface tension of the membrane.  The experimental basis for the theory comes from the work of Rozovsky et al in \cite{rozovsky2005formation}, in which domain formation in a ternary mixture of sphingomyelin, DOPC, and cholesterol is observed for a vesicle adhered to a substrate structure.
To study the relation between an increase in surface tension and the morphological transitions on the membrane plane,  
a  coupling between the local composition and the local curvature was proposed in \cite{Komura:Langmuir:2006}. 
The authors consider a free energy framework and use an energy functional first introduced in \cite{Seul:Science:1995} to model phase separation of a di-block copolymer in a membrane allowing out of plane (bending) distortions  (see also \cite{kawakatsu1993phase2,kawakatsu1993phase,leibler87}).
\begin{figure}
\begin{center}
\includegraphics[height=2in]{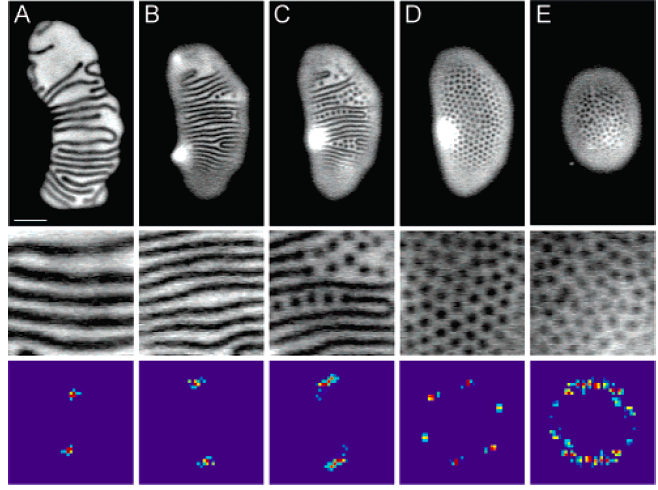}
\includegraphics[height=2in]{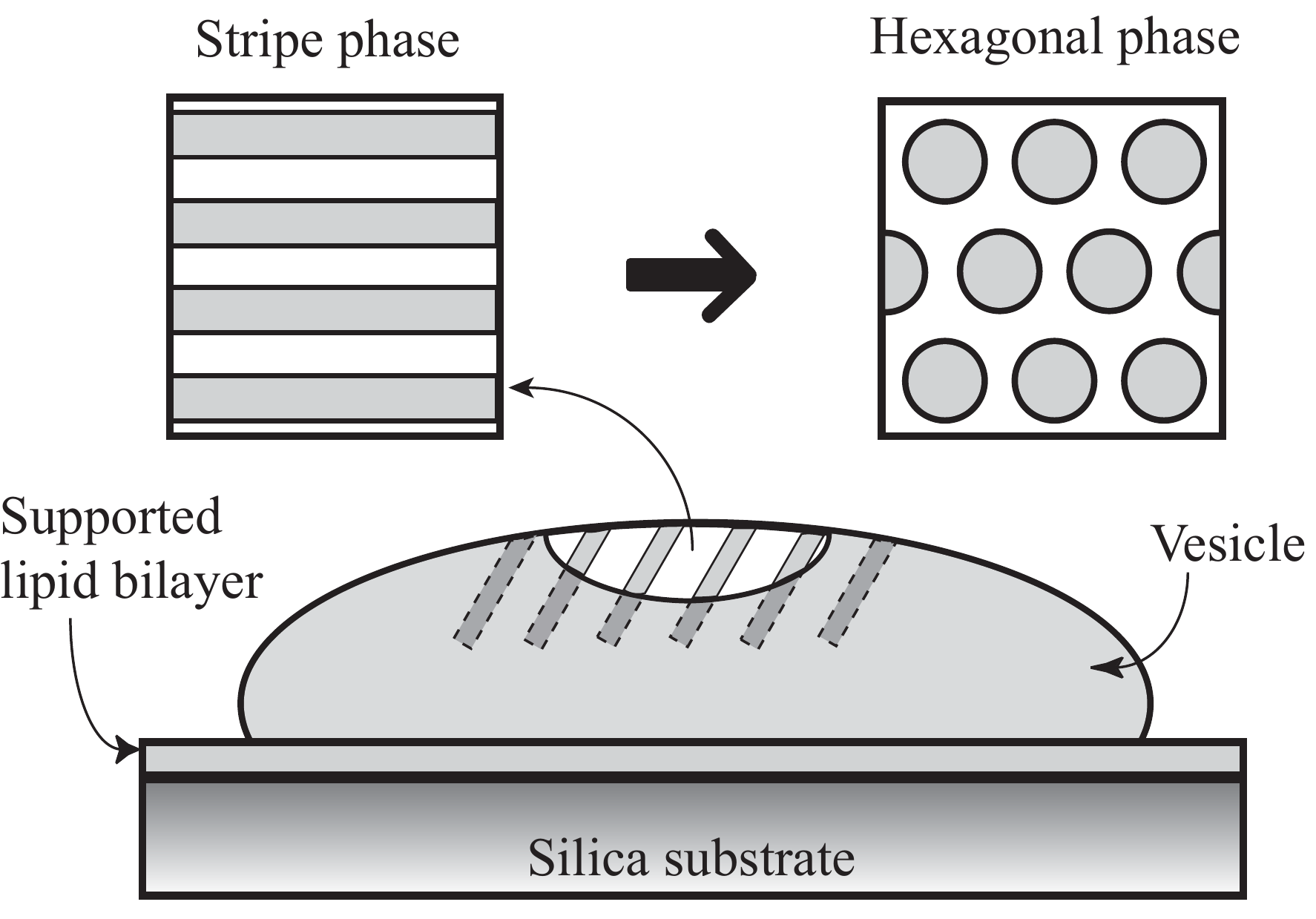}
\vskip 0.1in
\caption{Experimental and schematic representation of rafts. The first picture is reprinted with permission from S. Rozovsky, Y. Kaizuka, and J. T. Groves. Formation and spatio-temporal evolution of periodic structures
in lipid bilayers. J. Am. Chem. Soc., 127(1):36--37, 1 2005. Copyright 2005 American Chemical Society. It depicts epifluorescence microscopy images of phase separation in a vesicle composed of a mixture of sphingomyelin, DOPC, and cholesterol adhering to a supported lipid bilayer. Rafts, initially forming a stripe pattern evolve into a hexagonal array of circular domains as the vesicle changes shape. The last row depicts Fourier spectra for the ordered regions depicted in the middle row.
 The second picture, showing a schematic representation of the same process, is reprinted with permission from S. Komura, N. Shimokawa, and D. Andelman. Tension-induced morphological transition in mixed lipid
bilayers. Langmuir, 22:6771--6774, 2006. Copyright 2006 American Chemical Society.
}
\label{experiment1}
\end{center}
\end{figure}

 Similar to the classic Ginzburg-Landau models, the system is described in terms of an order parameter $u$ that may, for instance, model the relative composition of the lipids and cholesterol on the membrane plane.  However, in addition to the familiar Cahn-Hilliard/Modica-Mortola energy (see \cite{modica1987gradient}),
\beq
\label{ACFunc}
\mathcal{A}_\eps[u] := \int_\Omega \left( \frac{1}{\eps} W(u) + \eps |\nabla u|^2 \right) dx,
\eeq
that models line tension between domains and represents `short-range' interactions and whose minimization drives the system to evolve into $A$ rich and $B$ rich phases (corresponding to $u = \alpha$ or $u = \beta$, minima of a double-well potential $W$),  there are additional `forces' that prevent large domains of homogeneous concentration.  
In \cite{Seul:Science:1995} Seul and Andelman proposed a nonlocal contribution to the energy by considering an energy functional that  takes into account the bending energy of the membrane, and couples it to the value of the order parameter. The idea is that surface tension associated with a slightly curved membrane influences the localization of phases as the geometry of the lipids has an effect on the preferred placement on the membrane.  Similarly, the geometry of the membrane may adapt to that of the molecules.  The resulting energy has the form
\beq
\label{mainEnergy}
\mathcal{E}[\phi, h] = \int_D \left(f(\phi) + \frac{1}{2} b |\nabla \phi|^2 + \frac{1}{2} \sigma |\nabla h|^2 + \frac{1}{2} \kappa (\Delta h)^2 + \Lambda \phi \Delta h \right) d\bar{x}.
\eeq
Here $D := \{ Lx: x \in \Omega\}$ is the domain with the characteristic size $L$, $\phi$ is the order parameter, $h$ represents the height profile of the membrane, $f(\phi) := \frac{a_2}{2} \phi^2 + \frac{a_4}{4} \phi^4$, where $a_2, a_4$ are constants, $b > 0$ is related to the line tension between different domains, $\sigma > 0$ and $\kappa >0$ are the surface tension and bending rigidity of the membrane, respectively, and $\Lambda$ is the composition-curvature coupling constant.

\begin{table}
\label{parameterValuesTable}
\begin{center}
\begin{tabular}{|c|c|c|c|}
\hline
Symbol  & Description &  Value  \\
\hline
$a_4$ &  &  $10^{-5 }J/m^2$ \\
\hline
$b$ &  line tension &  $  5 \times 10^{-19} J$ \\
\hline
$\sigma$ & surface tension &  $5\times 10^{-6}$ to $10^{-4} J/m^2$ \\
\hline
$\kappa$ & bending rigidity of the membrane  &  $10^{-19} J$ \\
\hline $\Lambda$ & composition-curvature coupling constant &  $4.9 \times 10^{-12} J/m$ \\
\hline
\end{tabular}
\caption{Parameter descriptions and characteristic values, \cite{Komura:Langmuir:2006}.}
\end{center}
\end{table}

We note that several simplifying assumptions have been made in relation to the classical membrane energies (e.g. \cite{Canham70, Helfrich73}) or more recent multi-component biological membrane energies (e.g.  \cite{givli:2012}).  Rather than considering a closed hypersurface to represent the vesicle. We assume that the vesicle is almost flat and that its shape is described in terms of the distance, $h$, to the reference plane, $D \subset \mathbb{R}^2$. In addition, for simplicity, higher-order coupling terms between the composition and the curvature of the membrane are omitted.  There is no direct measure of the resulting single coupling parameter, $\Lambda$, but it can be fitted based on experimental data (see \cite{Komura:Langmuir:2006} for details).

Since minimizers of $\mathcal{E}$ satisfy the
Euler-Lagrange equations, we may consider the minimization problem for $\mathcal{E}[\phi, h]$ under the constraint, $\frac{\delta \mathcal{E}}{\delta h} = 0$.  
 Using the last equation to eliminate $h$ (see the Appendix) and rescaling
\begin{align}
u(x) = \phi(Lx), \quad \eps :=  \sqrt{\frac{\kappa}{L^2 \sigma}}, 
\quad q := 1 - \frac{b \sigma}{\Lambda^2},
\quad
W(u) := \frac{2\kappa}{\Lambda^2} f(u),
\mbox{\quad and\ }
\cFv_\eps := \frac{1}{\eps}\frac{2\kappa}{\Lambda^2 L^d} \mathcal{E},
\nonumber
\end{align}
one can reduce \eqref{mainEnergy} to
\beq
\label{mainFunct}
\cFv_\eps[u]  :=
\frac{1}{\eps} \int_\Omega \left(W(u) - u^2 + (1-q) \eps^2 |\nabla u|^2 + u \left( {\bf 1} - \eps^2 \Delta \right)^{-1} u \right)\, dx.
\eeq
Here $q$ is a constant parameter and the second order differential operator  ${\bf 1} - \eps^2 \Delta: H^2(\Omega) \rightarrow L^2(\Omega)$ is subject to Neumann boundary conditions. 
   A detailed derivation is given in the Appendix.  In addition, Table 1 lists typical values for the parameters.  Note that $\sqrt{\frac{\kappa}{\sigma}} \sim 10^{-7}m$, so the domain size of $10$ microns corresponds to $\eps \sim 10^{-2}$.  One may also easily check from the table that the relevant values of the parameter $q$ fall in the interval $(-1.1,1)$, and for fixed $b$ and $\Lambda$ correspond to varying the surface tension.    

Moreover, the line tension, surface tension, and the composition-curvature coupling constant are embedded in the effective parameter $q$.  To develop some intuition about the effect of varying $q$ we momentarily assume dependence only on a single direction and consider the energy of a single term in the Fourier series expansion of $u$ (see the Appendix),
$$
u(x) = \psi_n(x),\qquad x\in \Omega=(-1,1),
$$
where 
$$
\psi_n(x) := \cos(\lambda_n x)
\quad
\text{and}
\quad
\lambda_n =  2\pi n.
$$
Then, 
$$
({\bf 1} - \eps^2 \Delta)^{-1}  \psi_n = \frac{1}{1 + \eps^2 \lambda_n^2} \psi_n,
$$
and separating the potential term in the energy we have,
\begin{align}
\label{FExp}
\cF_\eps^*[\psi_n] & = \frac{1}{\eps} \int_{-1}^1 W(\psi_n) dx + \frac{1}{\eps} \int_{-1}^1 \left(- 1 + (1-q) \eps^2 \lambda_n^2 + \frac{1}{1+\eps^2 \lambda_n^2} \right) \cos^2(\lambda_n x) dx \nonumber \\
& =  \frac{1}{\eps} \int_{-1}^1 W(\psi_n) dx  + \frac{1}{\eps} \cF_{q,n},
\end{align}
where
\beq
\cF_{q,n} := - 1 + (1-q) \eps^2 \lambda_n^2 + \frac{1}{1+\eps^2 \lambda_n^2}.
\eeq
For fixed $q>0$, the minimum of $\cF_{q,n}$ is achieved by
$$
\eps^2 \lambda_{n_*}^2 = \frac{1}{\sqrt{1-q}} - 1,
$$
with the corresponding energy
\beq
\cF_{q,n_*} =  -2 +2 \sqrt{1-q}  +q<0.
\eeq
As evident from the calculations above, the contribution to the full energy from $\cF_{q,n_*}$ becomes negative as $q$ increases from $0$ (corresponding to the weakening tension).  Hence, depending on the properties of the potential $W$ the functional may be unbounded from below.  A natural question is to understand this bifurcation as $q$ increases. This paper represents a step towards that goal. In particular, we show that for a standard family of double-well potentials (see Hypotheses \ref{assume1}), even if $q$ is positive, the energy is bounded from below and $\Gamma$-converges to the perimeter functional for $q$ sufficiently small. 
Since the minimizers of the limiting energy have minimal interfaces, the physical interpretation is that for $L^2 \gg \kappa / \sigma$, $(\eps \ll 1)$ raft microdomains are not formed in this regime.  If the surface tension is too small and the functional is unbounded from below as $\eps \rightarrow 0$, different mathematical methods will have to be used to study the formation of raft-like microdomains (e.g. \cite{mizel:1998,peletier:1997}).

We remark that when $q \le 0$ the $\Gamma$-convergence to the perimeter functional can be proved under weaker conditions on the potential.  In that case the functional is nonnegative (this can be seen from the reformulation of the problem presented in \eqref{mainFunctNew}). The $\Gamma$-convergence of similar energies has been considered before (e.g. \cite{hilhorst:2002,baia2013coupled}), however there are some differences with the functional \eqref{mainFunctNew}   (for example when $q=0$) and will be addressed in a separate paper. 

Finally, we observe that in our context the relevant physical dimension is $d=2$, although the analysis presented here is carried out in arbitrary dimension $d\ge 2$.

\section{Preliminaries, Notation, and Statement of Results}
\label{prelimSection}

A natural mathematical framework for studying the asymptotic behavior of the family of functionals \eqref{mainFunct} is the notion of $\Gamma$-convergence introduced by De Giorgi in \cite{giorgi1975sulla} (see also \cite{braides2002gamma,dal1993introduction}).  In a general metric space setting the definition is given below.
\begin{definition}
Let $(Y, d)$ be a metric space and consider a sequence $\{\cF_n\}$ of functionals $\cF_n$: $Y \rightarrow [-\infty, \infty]$.  We say that $\{\cF_n\}$ $\Gamma$-converges to a functional $\cF: Y \rightarrow [-\infty, \infty]$ if the following properties hold:
\begin{enumerate}
\item (Liminf Inequality) For every $y \in Y$ and every sequence $\{y_n\} \subset Y$ such that $y_n \rightarrow y$, 
\beq
\cF[y] \le \liminf_{n\rightarrow \infty} \cF_n[y_n]. \nonumber
\eeq
\item (Limsup Inequality) For every $y \in Y$ there exists $\{y_n\} \subset Y$ such that $y_n \rightarrow y$ and
\beq
\limsup_{n\rightarrow \infty} \cF_n[y_n] \le \cF[y]. \nonumber
\eeq
\end{enumerate}
The functional $\cF$ is called the $\Gamma$-limit of the sequence $\{\cF_n\}$.
\end{definition}
A key property of $\Gamma$-convergence is the fact that, under appropriate compactness conditions, the sequence of minimizers of the functionals $\cF_n$ converge to a minimizer of the limiting functional $\cF$.
Moreover,  one can show that the isolated local minima of the $\Gamma$-limit $\cF$ persist under small perturbations (see \cite{kohn1989local,dal1993introduction}).

The problem of finding a characterization of the $\Gamma$-limit of \eqref{mainFunct} has been considered in the one-dimensional setting by Ren and Wei in \cite{ren2004soliton}, but in a different parameter regime.  Due to the different scaling of the terms, the technique used in that paper is not applicable to our case.
 Recall that the last term in \eqref{mainFunct} renders the problem nonlocal. 
A local approximation of \eqref{mainFunct} was studied in \cite{chermisi2010singular} and \cite{zeppieri41asymptotic}.
 We refer to the derivation of \eqref{AppEn} in the Appendix for the precise connection between the models. Qualitative properties of local minimizers of the local approximation model have been studied extensively to explain the formation of periodic layered structures (see \cite{bonheure:2003,coleman:1992,mizel:1998,peletier:1997}).

We now give the precise formulation of our results.  Let $\Omega \subset \R^d$, $d\ge 2$, be an open, bounded set of class $C^2$, and let $W$ be a twice continuously differentiable double-well potential defined on the real line. 
We make the following hypotheses on $W$.
\begin{assumptions}
\label{assume1}
\mbox{ }
\begin{enumerate}
\item $W(s) > 0$ if $s \ne \pm 1$.
\item $W(\pm 1) = 0$.
\item There exists $c_w > 0$ such that
$
W(s) \ge c_w (s \mp 1)^2 \mbox{ for } \pm s \ge 0.
$
\item There exist constants $K_w$, $C_w > 0$ such that
$
|W'(s)| \le C_w  \sqrt{W(s)}
$
and $|W''(s)| \le K_w$ for all $s \in \R$.
\end{enumerate}
\end{assumptions}
\begin{remark}
\label{remark1}
Note that conditions 3 and 4 imply that $W$ has quadratic growth at infinity.
\end{remark}
For the purposes of our analysis it will be convenient to rewrite the functional $\cFv_\eps$ as follows. Given $u\in W^{1,2}(\Omega)$, we define $v\in W^{3,2}(\Omega)$ via
\begin{align}
- \eps^2 \Delta v + v = u \text{ in } \Omega  \text{ \qquad and\qquad  }  \frac{\partial v}{\partial \n} = 0 \text{ on } \partial \Omega, \nonumber
\end{align} 
where $\n$ denotes the outward unit normal to $\partial\Omega$, and use the abbreviatory notation $v := ({\bf 1} - \eps^2 \Delta)^{-1} u$. Integrating by parts we obtain 
\begin{align}
\label{uv}
\cFv_\eps[u]
& =  \int_\Omega \left(\frac{1}{\eps} W(u) - \eps q |\nabla u|^2 + \eps^3 (\Delta v)^2 + \eps^5 |\nabla \Delta v|^2 \right) dx \nonumber \\
& =   \int_\Omega \left(\frac{1}{\eps} W(u) - \eps q |\nabla v|^2 + (1-2q) \eps^3 (\Delta v)^2 + (1-q) \eps^5 |\nabla \Delta v|^2 \right) dx  . \nonumber
\end{align} 
Hence, we may also view $\cFv_\eps$ as $\cF_\eps[v]$ with $\cF_\eps: L^{2}(\Omega) \rightarrow (-\infty,\infty]$ given by


\beq
\label{mainFunctNew}
\cF_\eps[v]  :=
\left\{\begin{array}{ll}
\cF_\eps[v; \Omega]   &\mbox{if  }  v \in W^{3,2}(\Omega),  \frac{\partial v}{\partial \n}=0\text{\ on\ }\partial\Omega,\\
+\infty & \text{otherwise},
\end{array}\right.
\eeq
where
\beq
\label{mainFuncFull}
\cF_\eps[v; A] = \int_A \left(\frac{1}{\eps} W(-\eps^2 \Delta v + v) - \eps q |\nabla v|^2 + (1-2q) \eps^3 (\Delta v)^2 + (1-q) \eps^5 |\nabla \Delta v|^2 \right) dx \nonumber
\eeq
for every open set $A \subset \Omega$.
\begin{remark}
Observe that if $v \in W^{3,2}(\Omega)$ does not satisfy Neumann boundary conditions on $\partial \Omega$, then $\cF_\eps[v; \Omega] < \cF_\eps[v] = \infty$.
\end{remark}

\begin{definition}
\label{mainDefs}
Given a vector $\nu \in \mathbb{S}^{d-1}$ ($d-1$ dimensional unit sphere), let  $\{\nu_1, \cdots, \nu_{d-1}, \nu\}$  be an orthonormal basis of $\mathbb{R}^d$.  We will denote by $Q_\nu$ an open unit cube centered at the origin with two of its faces normal to $\nu$, i.e.,
\beq
Q_\nu := \left\{ x \in \R^d: |x \cdot \nu| < \frac{1}{2}, \ |x \cdot \nu_i| < \frac{1}{2}, \ i = 1, \dots, d-1\right\}. \nonumber
\eeq
 If $x_0 \in \R^d$ and $r>0$, then $Q_\nu(x_0, r) := x_0 + rQ_\nu$.  If $\{\nu_1, \cdots, \nu_{d-1}, \nu\}$ is the canonical basis, we drop the dependence on $\nu$, i.e., $Q(x_0, r) := x_0 + r (-1/2, 1/2)^d = x_0 + rQ$, where $Q$ is the open unit cube centered at the origin with faces normal to the coordinate axes.

Define the admissible set to be
\begin{align}
\cA_\nu  :=  \{ v \in W^{3,2}_{loc}(\R^d): \, & v = 1 \mbox{ in a neighborhood of } x \cdot \nu = -1/2, \nonumber \\
&  v = -1 \mbox{ in a neighborhood of } x \cdot \nu = 1/2, \nonumber \\
& v(x) = v(x+ \nu_i) \mbox{ for all } x \in \R^d, \ i = 1, \dots, d-1 \}, \nonumber
\end{align}
and set
\beq
\label{eq:md}
m_{d}  := \inf  \{ \cF_{\eps}[v; Q_\nu]: 0< \eps \le 1, v \in \cA_\nu \} .
\eeq
\end{definition}
As we will see in the sequel (see \eqref{LimitFunctional}) the constant $m_d$ represents the surface energy density per unit area of the limit energy. The fact that $m_d$ is characterized by the cell problem \eqref{eq:md} is to be expected in this type of singular perturbations problems (see, e.g., \cite{baia2013coupled}, \cite{chermisi2010singular}, \cite{zeppieri41asymptotic}). As it turns out, in the case in which only first order derivatives are considered in the energy functionals, $m_d$ reduces to a one-dimensional geodesic distance between the wells for an appropriate metric involving the double-well potential $W$ (see \cite{fonseca1989gradient}). 

\begin{remark}
Since the gradient and Laplacian are invariant with respect to rotations, we can choose the coordinate system in such a way that the standard vector $e_d$ is parallel to $\nu$.  It follows that $m_d$ does not depend on $\nu$, and we abbreviate $\cA := \cA_{e_d}$.
\end{remark}
\begin{remark}
We will show in Proposition \ref{mdPosProp} that $m_d > 0$ if $q$ is sufficiently small.
\end{remark}
We introduce the  functional $\cF: L^2(\Omega) \rightarrow [0, +\infty]$,
\beq
\label{LimitFunctional}
\cF[v] := \left\{\begin{array}{ll}
m_d \mbox{Per}_\Omega(\{v=1\})  &\mbox{ if  }  v \in BV(\Omega; \{-1,1\}), \\
+\infty & \mbox{ if } v \in L^2(\Omega) \backslash BV(\Omega; \{-1,1\}).
\end{array}\right.
\eeq
Here $BV(\Omega; \{-1,1\})$ denotes the space of functions of bounded variation taking values in the set $\{-1,1\}$, (see the discussion at the end of the section).
The following theorems establish the $\Gamma$-convergence of $\cF_\eps$ to $\cF$, and ensures convergence of almost minimizers of $\cF_{\eps}$ to minimizers of $\cF$.
\begin{theo}
\label{CompactnessRn}
(Compactness)  Assume that $W \in C^2(\R)$ satisfies Hypotheses \ref{assume1}.  There exists $q_0 > 0$, depending only on the potential $W$, such that if $q < q_0$,
$\eps_n \rightarrow 0^+$ and  $\{v_n\} \subset W^{3,2}(\Omega)$ satisfies
\beq
\label{compUB}
\sup_n \cF_{\eps_n} [v_n] < \infty,
\eeq
then there exist a subsequence $\{v_{n_k}\}$ of $\{v_n\}$ and $v \in BV(\Omega;\{-1,1\})$ such that
\beq
\label{compConv}
v_{n_k} \rightarrow v \quad \mbox{ and } \quad \eps_{n_k}^2 \Delta v_{n_k} \rightarrow 0 \quad  \mbox{ in } L^2(\Omega).
\eeq
\end{theo}
\begin{theo}
\label{GamamConvThm}
Assume that $W \in C^2(\R)$ satisfies Hypotheses \ref{assume1}.  There exists $\bar{q} > 0$, depending only on the potential $W$ and $\Omega$, such that for all $0 < q < \bar{q}$ the following inequalities hold:
\begin{enumerate}
\item Liminf Inequality: For every sequence of positive real numbers $\eps_n \rightarrow 0$, for every $v \in L^2(\Omega)$, and for every $\{v_n\} \subset W^{3,2}(\Omega)$ such that  $v_n \rightarrow v$ in $L^2(\Omega)$,
\beq
\label{LimInf}
\liminf_{n \rightarrow \infty} \cF_{\eps_n}[v_n] \ge \cF[v].
\eeq
\item Limsup Inequality: For every $v \in L^2(\Omega)$ and for every sequence of positive real numbers $\eps_n \rightarrow 0$, there exists a sequence $\{v_n\} \subset W^{3,2}(\Omega)$  such that $v_n \rightarrow v$ in $L^2(\Omega)$ and
\beq
\label{Limsup}
\limsup_{n \rightarrow \infty} \cF_{\eps_n}[v_n] \le \cF[v].
\eeq
\end{enumerate}
\end{theo}
\begin{remark}
We remark that Theorem \ref{GamamConvThm} and the compactness property stated in Theorem \ref{CompactnessRn} have analogous formulations for the functional  $\cFv_\eps$ in \eqref{mainFunct}.  In particular, since for $v_n :=({\bf 1} - \eps^2 \Delta)^{-1} u_n$, $\cF_{\eps_n}[v_n] = \cFv_{\eps_n}[u_n]$, the compactness property follows from \eqref{compConv} due to the fact that $\sup_n \cF_\eps^*[u_n] < \infty$ implies that $u_{n_k} = -\eps_{n_k}^2 \Delta v_{n_k} + v_{n_k} \rightarrow v$ in $L^2(\Omega)$.  Similarly, for $u_n \rightarrow v$ in $L^2(\Omega)$, inequalities \eqref{LimInf} and \eqref{Limsup} of Theorem \ref{GamamConvThm} hold with $\cF_{\eps_n}[v_n]$ replaced by $\cFv_{\eps_n}[u_n]$.
\end{remark}

We now give a proof of an elliptic regularity result used in the sequel.
\begin{proposition}
\label{EllipticEstProp}
If $\Omega$ has a piecewise $C^2$ boundary,   then there exists a constant $C(\Omega)$, depending on $\Omega$, such that
\beq
\label{EllEst}
||\nabla^2 v||_{L^2(\Omega)}^2 \le 3 ||\Delta v||_{L^2(\Omega)}^2 + C(\Omega) ||v||_{L^2(\Omega)}^2
\eeq
for all $v \in W^{2,2}(\Omega)$ such that $\frac{\partial v}{\partial \n} = 0$ on $\partial \Omega$.
\end{proposition}
\begin{proof}
Theorem 3.1.1.2 from \cite{grisvard2011elliptic} yields
\beq
\label{GrisvardIneq}
 \int_\Omega |\nabla^2 v|^2  dx   \le  \int_\Omega |\Delta v|^2 dx + C_1(\Omega) \int_{\partial \Omega} |\nabla v|^2 dx
\eeq
for all $v \in W^{2,2}(\Omega)$ with $\frac{\partial v}{\partial \n} = 0$ on $\partial \Omega$, where the constant $C_1(\Omega)$ depends only on the curvature of $\partial \Omega$.
In turn, applying Theorem 1.5.1.10 from \cite{grisvard2011elliptic} to each component of $\nabla v$ we obtain
\beq
C_1(\Omega) \int_{\partial \Omega} |\nabla v|^2 dx \le \frac{1}{2}  \int_\Omega |\nabla^2 v|^2 dx + C_2(\Omega)  \int_\Omega |\nabla v|^2 dx \nonumber
\eeq
for some $C_2(\Omega) > 0$ and for all $v \in W^{2,2}(\Omega)$.  This, together with \eqref{GrisvardIneq}, reduces to
\beq
\label{GrisvardIneq2}
 \int_\Omega |\nabla^2 v|^2  dx   \le  2 \int_\Omega (\Delta v)^2 dx + 2 C_2(\Omega) \int_{\Omega} |\nabla v|^2 dx.
\eeq
Finally, using the Neumann boundary condition and integration by parts we conclude that
\beq
\label{GrisvardIneq3}
2 C_2(\Omega) \int_\Omega |\nabla v|^2 dx = 2C_2(\Omega) \int_\Omega (-\Delta v) v dx \le \int_\Omega (\Delta v)^2 dx + C(\Omega) \int_\Omega v^2 dx,
\eeq
where in the last step we also used Young's Inequality.  Inequalities \eqref{GrisvardIneq2} and \eqref{GrisvardIneq3} now imply \eqref{EllEst}.
\end{proof}

\begin{figure}
\begin{center}
\includegraphics[height=2in]{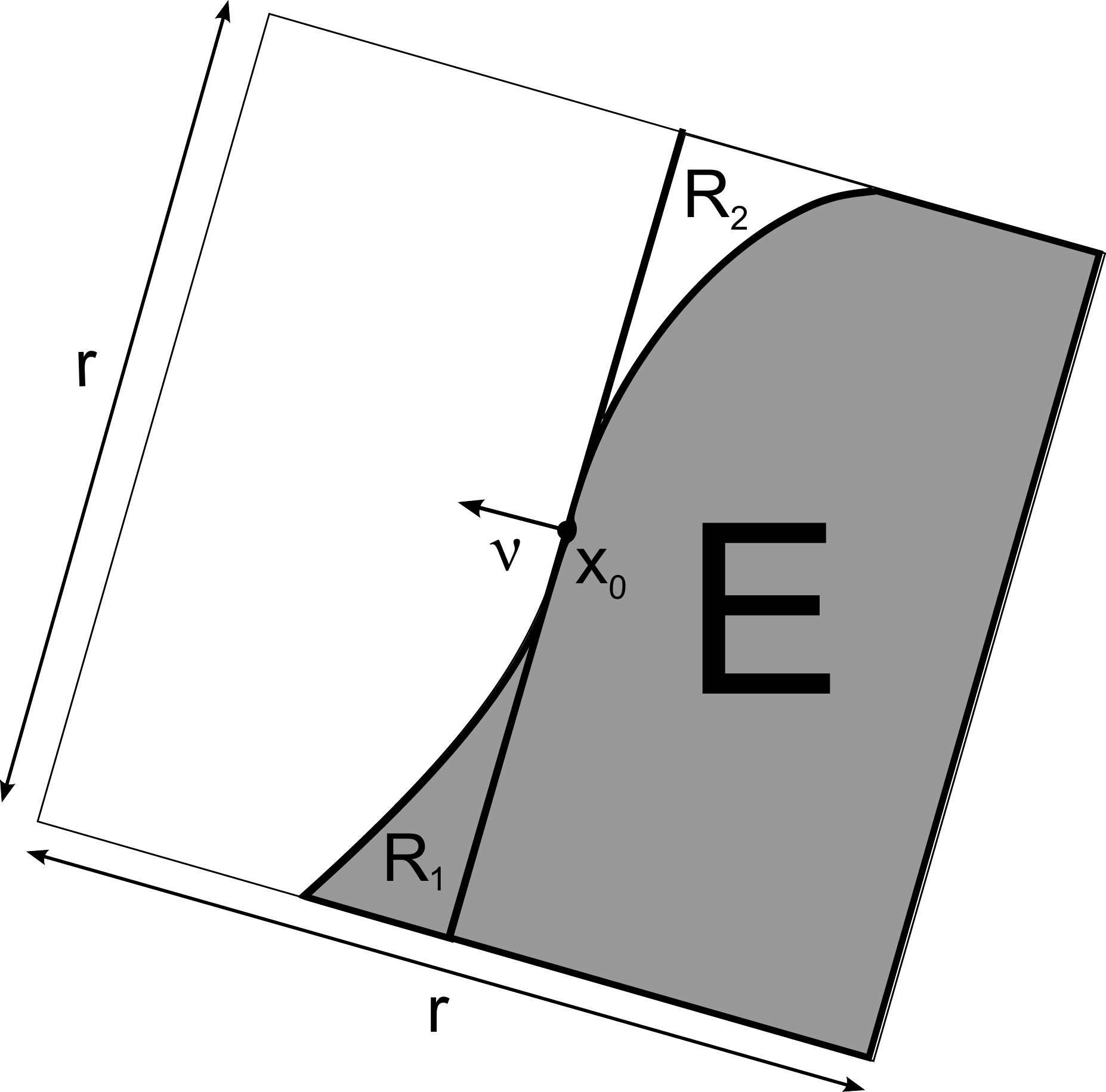}
\vskip 0.1in
\caption{The sets E (in grey), $R_1$, and $R_2$.}
\label{experiment}
\end{center}
\end{figure}

For the reader's convenience we end the section with a summary of  standard measure-theoretic results used in the remainder.
A key concept used in the development of the Liminf Inequality in Section \ref{LiminfSection} is that of a reduced boundary of the set $E := \{x \in \Omega: v(x) = 1\}$ associated to $v \in BV(\Omega; \{-1,1\})$.  We recall that $v \in L^1(\Omega)$ is said to be of bounded variation, $v \in BV(\Omega)$, if the generalized partial derivatives $D_i$ of $v$ in the sense of distributions are bounded Radon measures. In particular $BV(\Omega; \{-1, 1\})$ denotes functions of bounded variation taking values in the set $\{-1,1\}$, and $\text{Per}_\Omega(E) := |D \chi_{E}|(\Omega) < \infty$.

For sets of finite perimeter the reduced boundary $\partial^* E$ of $E$ is defined as the set of points $x_0 \in \text{spt}|D\chi_E| \cap \Omega$ such that the limit
\beq
\nu(x_0) := -\lim_{r \rightarrow 0^+} \frac{D \chi_E(B_r(x_0))}{|D \chi_E|(B_r(x_0))} \nonumber
\eeq
exists and satisfies $|\nu(x_0)| = 1$. Here $B_r(x_0)$ is the open ball of radius $r$ centered at $x_0$.   For $x_0 \in \partial^* E$ the vector $\nu(x_0)$ is called the generalized outer unit normal to $E$.  In particular, by Theorem 3.59 from {\cite{ambrosio2000functions}, $|D \chi_E| = \cH^{d-1} \mres \partial^* E$, and for $x_0 \in \partial^* E$,
\beq
\label{MT1}
\lim_{r \rightarrow 0^+} \frac{\cH^{d-1} (Q_{\nu}(x_0, r) \cap \partial^* E)}{r^{d-1}} = 1, \nonumber
\eeq
\beq
\label{MT2}
\lim_{r \rightarrow 0} \frac{1}{r^d} |R_1| = 0, \quad
\lim_{r \rightarrow 0} \frac{1}{r^d} |R_2| = 0,
\eeq
where
\beq
R_1 := \{x \in Q_\nu(x_0, r) \cap E: (x-x_0) \cdot \nu(x_0) > 0 \},
\eeq
\beq
R_2 := \{x \in Q_\nu(x_0, r)\backslash E: (x-x_0) \cdot \nu(x_0) < 0 \},
\eeq
and $|\cdot|$ denotes the Lebesgue measure in $\mathbb{R}^d$.
\section{Compactness}
In this section we prove the compactness Theorem \ref{CompactnessRn}. We  use the following interpolation inequality.
\begin{proposition}
\label{interpolProp}
Let $A \subset \mathbb{R}^d$ be  a bounded open set in $\R^d$.  Assume, in addition, that either $A$ has a $C^1$ boundary or that $A$ can be written as the union of  finitely many pairwise disjoint open rectangles and a set of Lebesgue measure zero. Then there exist a constant $q_* \in (0,1)$, independent of $A$, and $\eps_0 = \eps_0(A, q_*) > 0$ such that
\beq
\label{interpol}
q_* \int_A \eps |\nabla v|^2 dx \le \int_A \left( \frac{\W(v)}{\eps} + \eps^3 |\nabla^2 v|^2 \right) dx
\eeq
for every $\eps \in (0, \eps_0)$ and $v \in W^{2,2}(A)$.
\end{proposition}
\begin{proof}
See Theorem 1.2 in \cite{chermisi2010singular}.
\end{proof}
For every open set $A \subset \Omega$, $v \in W^{3,2}(\Omega)$, and $\eps > 0$, define the functional
\beq
\label{cIDef}
\cI_\eps[v;A] := \int_A \left( \frac{1}{\eps} W(v) + \eps |\nabla v|^2 + \eps^3 |\nabla^2 v|^2 + \eps^5 |\nabla \Delta v|^2 \right) dx. \nonumber
\eeq
\begin{remark}
We note that in the energy $\cF_\eps[v]$ the potential $W$ acts on $u$, which is related to $v$ through the condition $u = -\eps^2 \Delta v + v$, while in $\cI_\eps[v]$ the potential acts on $v$. Hence $\cF_\eps$ differs from the standard Cahn-Hilliard energies involving solely  the potential $W(v)$. In addition, the second order term in $\cF_\eps[v]$ involves the Laplacian $\Delta v$, while the second order term in $\cI_\eps[v]$ involves the Hessian $\nabla^2 v$.
\end{remark}
Next, we prove a result that will be useful to bound the energy from below and to obtain compactness of energy bounded sequences (see Theorem \ref{CompactnessRn}).
\begin{proposition}
\label{LBThm}
Let $K_w, C_w, c_w, q_*, \eps_0 > 0$ be the constants given in Hypotheses \ref{assume1} and Proposition \ref{interpol}. Then there exist $q_0 > 0$, depending only on $K_w, C_w, q_*$  (see \eqref{qBar}), and $\eps_1 > 0$, depending only on $C_w$, such that for every $0 < q \le q_0$, $v \in W^{3,2}(\Omega)$, and $0 < \eps < \eps_1$,
\beq
\label{FLB}
\cF_{\eps}[v]   \ge  q \, \cI_{\eps}[v; \Omega]  - \frac{12q}{q_*} C(\Omega) \eps^3 |\Omega|
\eeq
for some constant $C(\Omega) > 0$.
\end{proposition}

\begin{proof}
If $v$ does not satisfy $\frac{\partial v}{\partial \n} = 0$ on $\partial \Omega$ then $\cF_\eps[v] = \infty$ and there is nothing to prove.  Otherwise,
fix $0 < \theta \le 1$. 
Using Taylor's formula for $W$ and the fact that $W''$ is bounded by Hypotheses \ref{assume1}, yields
\begin{align}
\label{24}
\cF_{\eps}[v] & = \cF_{\eps}[v; \Omega] = \int_\Omega \left( \frac{1}{\eps} W(-\eps^2 \Delta v + v) - \eps q |\nabla v|^2 + (1-2q) \eps^3 (\Delta v)^2 + (1-q) \eps^5 |\nabla \Delta v|^2 \right) dx \nonumber \\
& \ge  \int_\Omega \Big( \frac{\theta}{\eps}  W(v) - \theta W'(v)\eps \Delta v -  \eps q |\nabla v|^2 + \Big(1- 2q - \frac{\theta}{2} K_w \Big) \eps^3 (\Delta v)^2 + (1-q) \eps^5 |\nabla \Delta v|^2  \Big) dx. \nonumber \\
\end{align}
By Young's Inequality and the condition $|W'(s)| \le C_w \sqrt{W(s)}$ from Hypotheses \ref{assume1}, we have
\beq
\label{25}
W'(v)  \Delta v \le \frac{1}{2 \eps^2 C_w^2} (W'(v))^2 + \frac{\eps^2}{2} C_w^2 (\Delta v)^2 \le \frac{1}{2\eps^2} W(v) + \frac{\eps^2}{2} C_w^2 (\Delta v)^2.
\eeq
Substituting \eqref{25} into \eqref{24} implies
\beq
\label{LBInt}
\cF_{\eps}[v]  \ge \int_\Omega \left( \frac{\theta}{2\eps} W(v) - \eps q |\nabla v|^2  + \left(1 - 2q - \frac{\theta}{2} K_w - \frac{\theta}{2} C_w^2 \right)  \eps^3 (\Delta v)^2  + (1-q) \eps^5 |\nabla \Delta v|^2  \right) dx. \nonumber
\eeq
Multiplying \eqref{interpol}, with $A=\Omega$, by $2q/q_*$ and using it in the previous inequality gives
\begin{align}
\label{FLB3}
\cF_{\eps}[v]   & \ge  \int_\Omega \Bigg( \left(\frac{\theta}{2} - \frac{2q}{q_*} \right) \frac{1}{\eps} W(v) + \eps q |\nabla v|^2 +  \left(1 - 2q - \frac{\theta}{2} K_w - \frac{\theta}{2} C_w^2 \right)  \eps^3 (\Delta v)^2 \nonumber \\
& -  \frac{2q \eps^3}{q_*} |\nabla^2 v|^2 + (1-q) \eps^5 |\nabla \Delta v|^2   \Bigg) dx. 
\end{align}
Fix $\delta > 0$.  Using Proposition \ref{EllipticEstProp} we get
\begin{align}
\cF_\eps[v] & \ge \int_\Omega \Bigg( \left(\frac{\theta}{2} - \frac{2q}{q_*} \right) \frac{1}{\eps} W(v) - \left( \delta + \frac{2q}{q_*}\right) \eps^3 C(\Omega) v^2 + \eps q |\nabla v|^2  \nonumber \\
& + \left(1-2q-\frac{\theta}{2} K_w - \frac{\theta}{2} C_w^2 - \frac{6q}{q_*} - 3 \delta \right) \eps^3 (\Delta v)^2 + \delta \eps^3 |\nabla^2 v|^2  + (1-q) \eps^5 |\nabla \Delta v|^2 \Bigg). \nonumber
\end{align}
Finally, it follows from Hypotheses \ref{assume1} that $W(s) \ge (c_w/4) s^2$ for $|s| \ge 2$.
Hence
\begin{align}
\cF_\eps[v] & \ge \int_\Omega \Bigg( \left[ \frac{\theta}{2} - \frac{2q}{q_*} - \eps^4 \frac{4 C(\Omega)}{c_w} \left( \delta + \frac{2q}{q_*}\right) \right] \frac{1}{\eps} W(v) + \left(1 - 2q - \frac{\theta}{2} K_w - \frac{\theta}{2} C_w^2 - \frac{6q}{q_*} - 3 \delta \right) \eps^3 (\Delta v)^2  \nonumber \\
& + \eps q |\nabla v|^2 + \delta \eps^3 |\nabla^2 v|^2 + (1-q) \eps^5 |\nabla \Delta v|^2 \Bigg) dx - 4 \left( \delta + \frac{2q}{q_*} \right) \eps^3 C(\Omega) |\Omega|. \nonumber
\end{align}
Choosing $\delta := \frac{q}{q_*}$, $\theta := \frac{8q}{q_*}$, $\eps_1 := \min\left\{\eps_0, \left( \frac{c_w}{12 C(\Omega)} \right)^{1/4}\right\}$ and 
\beq
\label{qBar}
q_0 := \frac{q_*}{2q_* + 4K_w + 4 C_w^2 + 10}
\eeq
yields \eqref{FLB}.
\end{proof}
We now prove that for $q$ sufficiently small the ``cell'' energy is positive.  
\begin{proposition}
\label{mdPosProp}
Let $m_d$ be defined in \eqref{eq:md} and let $q_0$ be as in Proposition \ref{LBThm}.  Then $m_d > 0$ for every $
0 < q <  q_0. 
$
\end{proposition}
\begin{proof}
Without loss of generality we may assume that the infimum in the definition of $m_d$ is taken over $0 < \eps < \eps_0$.  
The result of the proposition then follows if we show that
\beq
\label{mdPos}
\inf \left\{ \int_Q \left( \frac{W(v)}{\eps} + \eps |\nabla v|^2 \right) dx: 0 < \eps < \eps_0, v \in \cA \right\} > 0.
\eeq
Indeed, let $v \in \cA$.  Since $v$ satisfies periodic boundary conditions on $Q$, integration by parts yields
\beq
\label{elliptEq}
||\nabla^2 v||_{L^2(Q)}^2 = ||\Delta v||_{L^2(Q)}^2.
\eeq
Repeating the proof of Proposition \ref{LBThm}  with $Q$ instead of $\Omega$ and using \eqref{elliptEq} in \eqref{FLB3}, we obtain
\beq
\label{cFLB5}
\cF_{\eps}[v; Q] \ge q\, \cI_{\eps}[v; Q] \ge q \int_Q \left( \frac{W(v)}{\eps} + \eps |\nabla v|^2 \right) dx \nonumber
\eeq
if $q \le q_0$.
To prove \eqref{mdPos} we follow \cite{fonseca1989gradient}. In particular, for $v \in \cA$,
\beq
\label{mmtrick}
 \int_Q \left( \frac{W(v)}{\eps} + \eps |\nabla v|^2 \right) dx \ge 2 \int_Q \sqrt{W(v)} |\nabla v| dx \ge \int_{Q'} \int_{-1/2}^{1/2} \sqrt{W(v)} \left| \frac{\partial v}{\partial x_d} \right| d x_d d x',
\eeq
where $Q' := (-1/2, 1/2)^{d-1}$.  Since $v(x', \pm 1/2) = \pm 1$ a change of variables yields
\beq
 \int_{Q'} \int_{-1/2}^{1/2} \sqrt{W(v)} \left| \frac{\partial v}{\partial x_d} \right| d x_d \, d x' \ge \int_{-1}^{1} \sqrt{W(s)}  \, ds. \nonumber
\eeq
Using this lower bound in \eqref{mmtrick} and taking the infimum over $v \in \cA$ and $0 < \eps < \eps_0$ gives \eqref{mdPos}.
\end{proof}
\begin{proof}[Proof of Theorem \ref{CompactnessRn}]
By Proposition \ref{LBThm} and \eqref{compUB} 
\beq
\label{IFuncBound}
\sup_n \ \cI_{\eps_n}[v_n; \Omega]< \infty.
\eeq
Hence, the Modica-Mortola energy, $\cA_\eps[v_n]$, of $v_n$ defined in \eqref{ACFunc} is uniformly bounded from above.
The existence of some $v \in BV(\Omega; \{-1,1\})$ and a subsequence $\{v_{n_k}\}$ converging to $v$ in $L^1(\Omega)$ is well established for sequences of functions with uniformly bounded Modica-Mortola energy (see \cite{modica1987gradient}).

To show the convergence in $L^2(\Omega)$, we recall again that by Hypotheses \ref{assume1},  $W(s) \ge (c_w/4) |s|^2$ for $|s| \ge 2$, and hence for every measurable set $E \subset \Omega$,
\begin{align}
\label{EquiCont}
\int_E |v_n|^2 \ dx & =   \int_{\{y \in E: |v_n(y)| < 2\}} |v_{n}|^2 dx +  \int_{\{y \in E: |v_{n}(y)| \ge 2\}} |v_{n}|^2 dx \nonumber \\
& \le  4 |E| + \frac{4}{c_w} \int_{E}  W(v_{n}) \le 4 |E| +C(q)  \eps_n, \nonumber
\end{align}
where in the last step we used \eqref{IFuncBound}.  Therefore $\{|v_{n_k}|^2\}$ is equi-integrable, and convergence of $\{v_{n_k}\}$ to $v$ in $L^2(\Omega)$ is a consequence of Vitali's Convergence Theorem.

To prove \eqref{compConv}$_2$, note that \eqref{IFuncBound} implies  $\eps_n^2 \|\Delta v_n\|_{L^2(\Omega)} \le C(q) \eps_n^{1/2}$.  It follows that  $\eps_n^2 \Delta v_n \rightarrow 0$ in $L^2(\Omega)$.
\end{proof}

\section{Slicing Propositions}

The slicing arguments in the following propositions will be used in the proof of the Liminf Inequality. In what follows  we adopt the notation introduced in Definition \ref{mainDefs}. 
\begin{proposition}
\label{LayerProp}
There exists a constant $C(d)>0$ with the following property:
If $K > 0$, $\m \in \mathbb{N}$, and $v \in W^{3,2}(Q(x_0, r_0))$ are such that
\beq
\label{LayerPropAsm}
\cI_\eps[v; Q(x_0, r_0)] \le K
\eeq
for some $0 < \eps < \eps_1 := \frac{r_0}{4 \m \sqrt{C(d)}}$, then there exists $i \in \{1, \dots, \m\}$ (depending on $v$) such that 
\beq
\cF_\eps[v; Q(x_0, r)] \ge q \, \cI_\eps[v; Q(x_0, r)] -  \frac{q}{q_*} \frac{6K}{\m} \nonumber
\eeq
and
\beq
\cI_\eps[v; L] \le \frac{K}{\m}, \nonumber
\eeq
for all $r \in \left( \frac{r_0}{2} \left(1 + \frac{2i -1}{2\m} \right), \frac{r_0}{2} \left(1 + \frac{i}{\m} \right) \right)$ and all
$
0 < q < q_1
$,
where
\beq
L := Q\left(x_0,\frac{r_0}{2} \left( 1 + \frac{i}{\m} \right)\right) \backslash Q \left(x_0,\frac{r_0}{2} \left( 1 + \frac{i-1}{\m} \right)\right) \nonumber
\eeq 
and
\begin{equation*}
q_1 :=  \frac{q_*}{2q_* +4 K_w + 4C_w^2 + 3 C(d) + 1}.
\end{equation*}
\end{proposition}
\begin{proof}
For simplicity we will use the notation $Q(r) := Q(x_0, r)$.  
The following estimate is obtained from the proof of Lemma 9.2.3 in \cite{Jost}.  Let $0 < r_1 < r_2 < r_0$.  Then,
\beq
\label{elIneq}
\int_{\Q(r_1)} |\nabla^2 v|^2  \le C(d) \left( \int_{\Q(r_2)} |\Delta v|^2 \ dx + \frac{1}{(r_2 - r_1)^2} \int_{Q(r_2) \backslash Q(r_1)} |\nabla v|^2 \ dx \right).
\eeq
Given $\m \in\mathbb{N}$, we first partition the set $Q(r_0) \backslash Q(r_0/2)$
into $\m$ layers
\beq
L^i := Q\left(\frac{r_0}{2}\left(1+\frac{i}{\m}\right)\right)\setminus Q\left(\frac{r_0}{2}\left(1+\frac{i-1}{\m}\right)\right) , \quad i = 1, \dots, \m. \nonumber
\eeq
Since
\beq
\sum_{i=1}^\m \cI_\eps[v; L^i] \le \cI_\eps[v; Q(r_0)], \nonumber
\eeq
by \eqref{LayerPropAsm} there exists a layer $L^{i^*}$ satisfying 
\beq
\label{transLayerEst}
\cI_\eps[v; L^{i^*}] \le  \frac{1}{\m} \cI_\eps[v; Q(r_0)] \le \frac{K}{\m}.
\eeq
Fix $r \in \left(\frac{r_0}{2}\left(1+\frac{2i^*-1}{2\m}\right), \frac{r_0}{2}\left(1+\frac{i^*}{\m}\right)\right)$.
Choosing $r_1:=\frac{r_0}{2}\left(1+\frac{i^*-1}{\m}\right)$, $r_2:=r$ and applying estimate \eqref{elIneq} we obtain
\beq
\int_{Q(r_1)} |\nabla^2 v|^2  \ dx \le C(d) \ \left(  \int_{Q(r)} |\Delta v|^2 \ dx  + \frac{16\m^2}{r_0^2} \int_{L^{i^*}} |\nabla v|^2\ dx  \right). \nonumber
\eeq
Adding  $\int_{L^{i^*}} |\nabla^2 v|^2 \ dx$ to both sides and multiplying by $\eps^3$ yields, by \eqref{transLayerEst},
\begin{align}
\label{transLayerEstM}
\eps^3 \int_{Q(r)} |\nabla^2 v|^2 \ dx & \le C(d) \ \left(  \eps^3 \int_{Q(r)} |\Delta v|^2 \ dx   + \frac{16\m^2}{r_0^2} \eps^3 \int_{L^{i^*}} |\nabla v|^2 \ dx \right) + \eps^3 \int_{L^{i^*}} |\nabla^2 v|^2 \ dx \nonumber \\
& \le C(d) \left( \eps^3 \int_{Q(r)} |\Delta v|^2 \ dx  + \frac{16\m^2}{r_0^2} \eps^2 \frac{K}{\m} \right) + \frac{K}{\m}. \nonumber
\end{align}
Let $0 < \eps_1^2 = \frac{r_0^2}{16 \m^2 C(d)}$.  Then for $0 < \eps < \eps_1$ we have
\beq
\label{regularityRes}
\eps^3 \int_{Q(r)}  |\nabla^2 v|^2 \ dx  \le C(d) \ \eps^3 \int_{\Q(r)} |\Delta v|^2  \ dx  + \frac{2K}{\m}.
\eeq
Repeating the argument of the proof of Proposition \ref{LBThm} with $\theta := \frac{8q}{q_*}$ until \eqref{FLB3} and using  \eqref{regularityRes} multiplied by 3 in place of Proposition \ref{EllipticEstProp} yields
\begin{align}
\cF_{\eps}[v; Q(r)] & \ge  \int_{Q(r)} \Bigg(  \frac{2q}{q_*} \frac{1}{\eps} W(v) + \left(1 - 2q - \frac{4q}{q_*} K_w - \frac{4q}{q_*} C_w^2 - \frac{3 q}{q_*} C(d)  \right) \eps^3 |\Delta v|^2  \nonumber \\
& +  q \eps |\nabla v|^2 + \frac{q}{q_*} \eps^3 |\nabla^2 v|^2 + (1-q) \eps^5 |\nabla \Delta v|^2 \Bigg) dx - \frac{q}{q_*} \frac{6K}{\m}  \nonumber \\
& \ge  q \, \cI_{\eps}[v; Q(r)]  -  \frac{q}{q_*} \frac{6K}{\m}, \nonumber
 \end{align}
provided $
0 < q <  q_1 $. 
This completes the proof.
\end{proof}
\begin{proposition}
\label{glProp}
Let $k \in \mathbb{N}$, $\eps_n \rightarrow 0^+$, $\nu \in \mathbb{S}^{d-1}$, and $\{w_n\} \subset W^{3,2}(Q_\nu(0,1))$ be such that 
\beq
\lim_{n \rightarrow \infty} \int_{Q_\nu(0,1)} |w_n - v_0|^2 dx = 0, \nonumber
\eeq
and
\beq
\label{IBoundProp1}
\cI_{\eps_n}[w_n; \tilde{L}_k] \le \frac{C_0}{k}
\eeq
for all $n$ and some $C_0 > 0$, not dependent on $k$,
where
\beq
\label{u0Def}
v_0(y) := \left\{\begin{array}{ll}
1  &\mbox{ if  }  y \cdot \nu < 0, \\
-1 & \mbox{ if } y \cdot \nu > 0,
\end{array}\right. \nonumber
\eeq
and
\beq
\tilde{L}_k := Q_\nu(0,1) \backslash Q_\nu(0, 1- 1/(4\m)). \nonumber
\eeq
Then
\beq
\label{limInfinProp}
\cF_{\eps_n}[w_n; Q_\nu(0,1)] \ge m_d - \frac{C}{\m}, \nonumber
\eeq
where the constant $C$ does not depend on $k$.
\begin{proof}

We modify $\{w_n\}$ to belong to the admissible class $\cA_\nu$ without increasing the energy.
Given $\Psi \in C_c^\infty(\mathbb{R}^d)$, with $\mbox{supp}(\Psi) \subset B_1(0)$ and $\int_{\mathbb{R}^d} \Psi(y) dy = 1$, and $\eps>0$, consider the mollifier
\beq
\label{mol}
\Psi_\eps(y) := \frac{1}{\eps^d} \Psi\left(\frac{y}{\eps}\right)
\eeq
and
\beq
\varphi_n := v_0 \ast \Psi_{\eps_n}. \nonumber
\eeq
Note that $\varphi_n \in C^\infty(\R^d)$ and
\beq
\label{phiBounds}
||\varphi_n||_{L^\infty(\mathbb{R^d})} \le 1, \quad ||\nabla \varphi_n||_{L^\infty(\mathbb{R^d})} \le C \eps_n^{-1}, \quad ||\nabla^2 \varphi_n||_{L^\infty(\mathbb{R^d})} \le C \eps_n^{-2},
\quad ||\nabla^3 \varphi_n||_{L^\infty(\mathbb{R^d})} \le C \eps_n^{-3}.
\eeq
In addition,
\beq
\varphi_n(y) = \left\{\begin{array}{ll}
1  &\mbox{ if  }  y \cdot \nu < -\eps_n, \\
-1 & \mbox{ if } y \cdot \nu > \eps_n,
\end{array}\right. \nonumber
\eeq
and
\beq
\nabla^s \varphi_n(y) = 0 \mbox{ if } |y \cdot \nu| > \eps_n, \quad s = 1, 2, 3. \nonumber
\eeq
Hence for $\eps_n$ sufficiently small $\varphi_n \in \cA_\nu$.
We want to define a function $z_n$ to equal  $\varphi_n$ near the boundary of $\Q_\nu$ and $w_n$ away from the boundary. To be precise,  we first  partition the set
$\tilde{L}_k=
Q_\nu(0, 1) \backslash Q_\nu(0, 1-1/(4\m))
$
into $\lceil \eps_n^{-1} \rceil$ layers, 
\beq
L_n^i := Q_\nu\left(0, 1- \frac{i-1}{4\m  \lceil \eps_n^{-1} \rceil} \right) \Big\backslash Q_\nu\left(0, 1  - \frac{i}{4\m \lceil \eps_n^{-1} \rceil}\right), \quad i = 1, \dots, \lceil \eps_n^{-1} \rceil, \nonumber
\eeq
where $\lceil x \rceil$ is defined as the smallest integer not less than $x$.
Since both $w_n \rightarrow v_0$ in $L^2(\Q_\nu)$ and $\varphi_n \rightarrow v_0$ in $L^2(\Q_\nu)$, we have
\beq
 ||w_n - \varphi_n||_{L^2(\Q_\nu)}^2  \rightarrow 0 \mbox{ as } n \rightarrow \infty. \nonumber
\eeq
Note that $\cup_i L_n^i = \tilde{L}_k \subset Q_\nu(0,1)$ and that  $L_n^i$ are pairwise disjoint, so the sum over all of the layers is bounded by
\beq
 \sum_i \cI_{\eps_n}[w_n; L_{n}^{i}] +  \frac{\sum_i ||w_n - \varphi_n||_{L^2(L_{n}^{i})}^2}{ ||w_n - \varphi_n||_{L^2(\Q_\nu)}^2} \le \frac{C_0}{k} + 1.  \nonumber
\eeq
Since there are $\lceil \eps_n^{-1} \rceil$ layers,  for one of these layers, say $L_n := L_{n}^{i^*}$, it holds
\beq
\label{transLayerEst2}
 \cI_{\eps_n}[w_n; L_{n}] +  \frac{||w_n - \varphi_n||_{L^2(L_{n})}^2}{ ||w_n - \varphi_n||_{L^2(\Q_\nu)}^2} \le \left(\frac{C_0}{k} + 1\right) \, \eps_n.
\eeq
Define
\beq
z_n := \eta_{n} {w}_n + (1-\eta_{n}) \varphi_n, \nonumber
\eeq
where $\eta_{n}$ is a smooth function with support in $Q_\nu(0,1)$ such that
\beq
\eta_{n}(x) := \left\{\begin{array}{ll}
0 & \mbox { if } x \in Q_n^{out} := Q_\nu\left(0, 1\right) \backslash Q_\nu\left(0, 1  - \frac{i^*-1}{4\m \lceil \eps_n^{-1} \rceil}\right), \\
\in (0,1) & \mbox{ if } x \in L_{n}, \\
1 & \mbox{ if } x \in Q_n^{in} := Q_\nu \backslash (Q_n^{out} \cup L_n),
\end{array}\right. \nonumber
\eeq
and
\beq
\label{etaBounds}
||\nabla^s \eta_{n}||_{L^\infty(\Q_\nu)} = \mathcal{O}\left(\frac{\m^s}{\eps_n^s}\right), \quad s=1,2,3.
\eeq
Moreover,
\beq
\cF_{\eps_n}[z_n; Q_\nu] = \cF_{\eps_n}[\varphi_n; Q_n^{out}] +  \cF_{\eps_n} [z_n; L_n] + \cF_{\eps_n} [w_n; Q_n^{in}]. \nonumber
\eeq
We observe that since $\cF_{\eps_n}[w_n; Q_\nu \backslash Q_n^{in}]$ can be negative it is not necessarily true that
$
\cF_{\eps_n}[w_n; Q_n^{in}] \le \cF_{\eps_n}[w_n; Q_\nu]. 
$
Instead, we use \eqref{IBoundProp1} to control the negative terms to obtain
\begin{align}
\label{EnergyDecomp}
\cF_{\eps_n}[z_n; Q_\nu] & \le \cF_{\eps_n}[\varphi_n; Q_n^{out}] +  \cF_{\eps_n} [z_n; L_n] + \cF_{\eps_n} [w_n; Q_\nu] + q \int_L \eps_n |\nabla w_n|^2 dx \nonumber \\
& \le \cF_{\eps_n}[\varphi_n; Q_n^{out}] + \cF_{\eps_n} [z_n; L_n] + \cF_{\eps_n} [w_n; Q_\nu] + q\frac{C_0}{\m}.
\end{align}
Note that for $s = 1,2,3$,
\beq
\label{varphi1}
\eps_n^{2s-1} \int_{Q_n^{out}} |\nabla^s \varphi_n|^2 dx \le \eps_n^{2s-1} \frac{C}{\eps_n^{2s}} |\{x \in Q_n^{out}: \varphi_n \ne \pm 1\}| \le \frac{C}{\m}.
\eeq
In addition, by the continuity of $W$,
\beq
\label{varphi2}
\frac{1}{\eps_n} \int_{Q_n^{out}} W(-\eps_n^2 \Delta \varphi_n + \varphi_n) dx \le \frac{C}{\eps_n}  |\{x \in Q_n^{out}: \varphi_n \ne \pm 1\}| \le \frac{C}{\m}.
\eeq
Together \eqref{varphi1} and \eqref{varphi2} imply
\beq
\label{OutEst}
\cF_{\eps_n}[\varphi_n; Q_n^{out}] \le \frac{C}{\m}.
\eeq
To estimate $\cF_{\eps_n}[z_n; L_n]$, we first note that
\begin{align}
\partial_{x_i} z_n = \partial_{x_i} \eta_{n} (w_n - \varphi_n)+ \eta_{n} \partial_{x_i} w_n + (1-\eta_{n}) \partial_{x_i} \varphi_n, \nonumber
\end{align}
and
\begin{align}
\partial_{x_i x_k} z_n & =  \partial_{x_i x_k} \eta_{n} (w_n - \varphi_n) + \partial_{x_i} \eta_{n} \partial_{x_k} w_n + \partial_{x_k} \eta_{n}  \partial_{x_i} w_n + \eta_{n} \partial_{x_i x_k} w_n \nonumber \\
& -  \partial_{x_i} \eta_{n} \partial_{x_k} \varphi_n - \partial_{x_k} \eta_{n} \partial_{x_i} \varphi_n + (1 - \eta_{n}) \partial_{x_i x_k} \varphi_n. \nonumber
\end{align}
We use \eqref{transLayerEst2} to control the derivatives of $w_n$ in the transition region $L_{n}$.  From \eqref{phiBounds}, \eqref{transLayerEst2}, \eqref{etaBounds}, the expressions for the derivatives of $z_n$ and the fact that $||w_n - \varphi_n||_{L^2(Q)} \rightarrow 0$, we readily obtain the following bounds on the terms in $\cF_{\eps_n}[z_n; L_{n}] $,
\begin{align}
\label{nablaWTilde}
\eps_n \int_{L_n} |\nabla z_n|^2 dx & \le C \int_{L_n}  \left( \eps_n |\nabla \eta_n|^2 |w_n - \varphi_n|^2 + \eps_n \eta_n^2 |\nabla w_n|^2 + \eps_n (1-\eta_n)^2  |\nabla \varphi_n|^2 \right) dx \nonumber \\
& \le C \left( \frac{\m^2 \eps_n}{\eps_n^2} ||w_n - \varphi_n||_{L^2(L_n)}^2 + \left(\frac{C_0}{k} + 1\right) \eps_n + \frac{\eps_n}{\eps_n^2} |\{x \in L_n: \varphi_n \ne \pm 1\}| \right) \nonumber \\
& \le C \left( \m^2 \left(\frac{C_0}{k} + 1\right)   ||w_n - \varphi_n||_{L^2(L)}^2 +  \left(\frac{C_0}{k} + 1\right) \eps_n + \frac{\eps_n}{\m} \right) \le  \frac{C}{\m} 
\end{align}
for $n$ sufficiently large, where we used $\left| \{x \in L_{n}: |x \cdot \nu| < \eps_n\} \right|  = \mathcal{O}(\eps_n^2/\m)$.
Similarly,
\begin{align}
 \eps_n^3 \int_{L_n} |\nabla^2 z_n| dx & \le C \eps_n^3  \int_{L_n} \Big( |\nabla^2 \eta_n|^2 |w_n - \varphi_n|^2 + 2 |\nabla \eta_n|^2 |\nabla w_n|^2 + 2 |\nabla \eta_n|^2 |\nabla \varphi_n|^2 + \eta_n^2 |\nabla^2 w_n|^2 \nonumber \\
& + (1-\eta_n)^2 |\nabla^2 \varphi_n|^2 \Big) dx \le C \Bigg( \frac{\eps_n^3 \m^4}{\eps_n^4} \left(\frac{C_0}{k} + 1\right)  \eps_n  ||w_n - \varphi_n||_{L^2(L)}^2 +  \frac{\eps_n^2 \m^2}{\eps_n^2}  \left(\frac{C_0}{k} + 1\right)  \eps_n \nonumber \\
& +  \eps_n^3 \left(\frac{\m^2}{\eps_n^2} \frac{1}{\eps_n^2} + \frac{1}{\eps_n^4} \right) |\{x\in L_n : \varphi_n \ne \pm 1\}| + \left(\frac{C_0}{k} + 1\right)  \eps_n  \Bigg) \le \frac{C}{\m} \nonumber
\end{align}
for $n$ sufficiently large.
To bound the integral involving the potential $W$ we first remark that by Hypotheses \ref{assume1} (and Remark \ref{remark1}) $W$ grows quadratically at infinity.  Splitting the integral into  regions where  $|-\eps_n^2 \Delta z_n + z_n| \le 2$ and  $|-\eps_n^2 \Delta z_n + z_n| > 2$, we use the quadratic growth of  $W$  to obtain,
\begin{align}
\label{WBound}
& \left| \frac{1}{\eps_n} \int_{L_{n}} \W(-\eps_n^2 \Delta z_n + z_n)  dx \right|  \le   \frac{\sup_{|s| \le 2} W(s)}{\eps_n} | L_{n}| + \frac{C_w^2}{4 \eps_n}  \int_{L_{n}} (-\eps_n^2 \Delta z_n + z_n)^2  dx \nonumber \\
& \le  \frac{C}{\m} + \frac{C_w^2}{2}  \int_{L_{n}}  \eps_n^3 |\Delta z_n|^2 dx +  \frac{C_w^2}{2 \eps_n}   \int_{L_{n}}   z_n^2 dx 
 \le  \frac{C}{\m} + \frac{C_w^2}{2}  \int_{L_{n}}  \eps_n^3 |\Delta z_n|^2 dx + \frac{C_w^2}{\eps_n}   \int_{L_{n}}  (w_{n}^2 + \varphi_n^2) dx \nonumber \\
&  \le  \frac{C}{\m} + \frac{C_w^2}{2}  \int_{L_{n}}  \eps_n^3 |\Delta z_n|^2 dx +  \frac{C}{\eps_n}   \int_{L_{n}}   \W(w_{n}) dx   + \frac{C}{\eps_n }|L_n|  \le\left(\frac{C_0}{k} + 1\right)  \eps_n + \frac{C}{\m} \le \frac{C}{\m} 
\end{align}
for $n$ sufficiently large, where we again used \eqref{transLayerEst2}.
Analogous calculations are used to estimate \\ $\eps_n^5 \int_{L_n} |\nabla \Delta z_n|^2 dx$.
Combining estimates \eqref{OutEst}, \eqref{nablaWTilde}-\eqref{WBound} with \eqref{EnergyDecomp} completes the proof.
\end{proof}
\end{proposition}

\section{Proof of the Liminf Inequality}
\label{LiminfSection}

In this section we prove the Liminf Inequality of Theorem \ref{GamamConvThm}.  We use the blow-up method to reduce the problem to a unit cube, where we follow the general lines of  \cite{chermisi2010singular}.
In what follows we assume $q \le \min\{q_0, q_1\}$ (see Propositions \ref{LBThm} and \ref{LayerProp}).
Fix $\eps_n \rightarrow 0^+$ and $\{v_n\} \subset W^{3,2}(\Omega)$, $v_n \rightarrow v \in L^2(\Omega)$.
We may assume that
\beq
\label{F4B}
\liminf_{n \rightarrow \infty} \cF_{\eps_n}[v_n] < \infty,
\eeq
and we extract a subsequence $\{v_{n_k}\}$ of $\{v_n\}$ satisfying
\beq
\lim_{k \rightarrow \infty}  \cF_{\eps_{n_k}}[v_{n_k}]  = \liminf_{n \rightarrow \infty} \cF_{\eps_n}[v_n] < \infty. \nonumber
\eeq
By selecting a further subsequence, if necessary, we can assume that $\sup_k \cF_{\eps_{n_k}} [v_{n_k}] < \infty$ so that by Proposition \ref{LBThm},
\beq
\label{B}
\sup_k \cI_{\eps_{n_k}} [v_{n_k}; \Omega] =: K < \infty.
\eeq
Since $v_{n_k} \rightarrow v$ in $L^2(\Omega)$, Theorem \ref{CompactnessRn} implies that $v \in BV(\Omega; \{-1,1\})$.
Therefore,
\beq
\label{LimitV}
v = \chi_{E} - \chi_{\Omega \backslash E},
\eeq
where $\mbox{Per}_\Omega(E) < \infty$. In what follows, to simplify  notation we denote the subsequence of $\{ v_n\}$ extracted in \eqref{B} by $\{v_n\}$.

We first note that, due to \eqref{F4B} and \eqref{B}, the sequences of functions
\beq
f_n :=  \frac{1}{\eps_n} W(-\eps_n^2 \Delta v_n + v_n) - \eps_n q |\nabla v_n|^2 + (1-2q) \eps_n^3 |\Delta v_n|^2 + (1-q) \eps_n^5 |\nabla \Delta v_n|^2 \nonumber
\eeq
and
\beq
g_n :=  \frac{1}{\eps_n} W( v_n) + \eps_n  |\nabla v_n|^2 + \eps_n^3 |\Delta v_n|^2 +  \eps_n^5 |\nabla \Delta v_n|^2 \nonumber
\eeq
are bounded in $L^1(\Omega)$.  Consider the  signed Radon measures defined on Borel subsets of $\Omega$,
\beq
\lambda_n(B) := \int_B f_n \ dx,
\quad
\zeta_n(B) := \int_B g_n \ dx. \nonumber
\eeq
Up to subsequences, not relabeled, we may assume that there exist Radon measures $\lambda$, $\mu$, $\zeta$ such that
\beq
\lambda_n \mbox{ }   \rightharpoonup^* \lambda, \quad |\lambda_n|  \rightharpoonup^* \mu, \quad \zeta_n \rightharpoonup^* \zeta   \nonumber
\eeq
in the space $\mathcal{M}_b(\Omega)$ of all bounded signed Radon measures on $\Omega$  (see Proposition 1.202 in \cite{fonseca2007modern}), where $|\lambda_n|$ denotes the total variation of $\lambda_n$. We claim that $\lambda \ge 0$.

Suppose that $\lambda\neq 0$. By the Besicovitch Derivation Theorem (Theorem 1.155 in \cite{fonseca2007modern}), for $|\lambda|$-a.e. $x_0 \in \Omega$ 
\beq
\label{deriv}
\frac{d \lambda}{d|\lambda|}(x_0) =  \lim_{r \rightarrow 0^+} \frac{\lambda(\Q(x_0, r))}{|\lambda|(\Q(x_0, r))}  \in \R,
\eeq
where $|\lambda|$ is the total variation of $\lambda$. 
Fix any $x_0$ for which \eqref{deriv} holds and $|\lambda|(Q(x_0,r))>0$ for all $r>0$ sufficiently small. Let $\eta \in (0,1)$ and find $\bar{r}_\eta > 0$ such that
\beq
\label{a}
\frac{d \lambda}{d|\lambda|}(x_0) \ge \frac{\lambda(\Q(x_0, r))}{|\lambda|(\Q(x_0, r))}  - \eta
\eeq
for all $0 < r < \bar{r}_\eta$.  

Fix $0 < r_0 < \bar{r}_\eta$ and $\m \in \mathbb{N}$.  By Proposition \ref{LayerProp} for every $n$ there exists $i_n \in \{1, \dots, \m\}$ such that
\beq
\label{b}
\cF_{\eps_n}[v; Q(x_0, r)] \ge q \cI_{\eps_n}[v; Q(x_0, r)] - \frac{q}{q^*} \frac{6K}{\m}
\eeq
for all $r \in \left( \frac{r_0}{2} \left(1 + \frac{2i_n -1}{2\m} \right), \frac{r_0}{2} \left(1 + \frac{i_n}{\m} \right) \right)$ where $K$ is given in \eqref{B}.
Since $i_n \in \{1, \dots, \m\}$ for all $n$, there exists $i^{(1)} \in \{1, \dots, \m\}$ such that
$i^{(1)} = i_n$  for infinitely many $n$, say $n_l$, $l \in \mathbb{N}$.
Let $\m$ be so large that
\beq
\label{Qe}
\frac{q}{q_*} \frac{6K}{\m} \le |\lambda| ( Q(x_0, r_0/2 )) \eta
\eeq
and take
\beq
r_1 \in \left( \frac{r_0}{2} \left(1 + \frac{2i^{(1)} -1}{2\m} \right), \frac{r_0}{2} \left(1 + \frac{i^{(1)}}{\m} \right) \right) \nonumber
\eeq
such that $\mu(\partial Q(x_0, r_1)) = 0$.  Then by \eqref{a}, Corollary 1.204 in \cite{fonseca2007modern}, \eqref{b}, and \eqref{Qe}
\begin{align}
\frac{d \lambda}{d|\lambda|}(x_0) & \ge \frac{\lambda(\Q(x_0, r_1))}{|\lambda|(\Q(x_0, r_1))}  - \eta = \lim_{n \rightarrow \infty} \frac{\cF_{\eps_{n_l}}[v_{n_l}; Q(x_0, r_1)]}{|\lambda| (Q(x_0, r_0))} - \eta \nonumber \\
& \ge \liminf_{n \rightarrow \infty} \frac{ q \cI_{\eps_{n_l}}[v_{n_l}; Q(x_0, r_1)] - |\lambda|(Q(x_0, r_0/2) \eta}{|\lambda|(\Q(x_0, r_1))} - \eta \nonumber \\
& \ge - 2 \eta, \nonumber
\end{align}
where we used the fact that $r_0/2 < r_1$ so that $|\lambda|(Q(x_0, r_1)) \ge |\lambda| (Q(x_0, r_0/2))$.  Letting $\eta \rightarrow 0^+$ we conclude that $\frac{d \lambda}{d |\lambda|} (x_0) \ge 0$.

This shows that $\lambda \ge 0$.  In turn, by the Radon-Nikodym and Lebesgue Decomposition theorems (\cite{fonseca2007modern} Theorem 1.180) we can decompose
\beq
\lambda = \lambda_{ac} + \lambda_s, \nonumber
\eeq
where $\lambda_{ac} \ll \xi$, $\lambda_s \ge 0$, $\lambda_s \perp \xi$, with
\beq
\xi(B) := \cH^{d-1}(B \cap \partial^* E), \quad B \subset \Omega \mbox{ Borel}. \nonumber
\eeq
We claim that for $\cH^{d-1}$-a.e. $x_0 \in \Omega \cap \partial^* E$,
\beq
\label{C}
\frac{d\lambda_{ac}}{d \xi} (x_0) \ge m_d,
\eeq
where $m_d$ is the constant defined in \eqref{eq:md}.
Observe that if \eqref{C} holds, then, since $\lambda_s \ge 0$,
\begin{align}
\lim_{n \rightarrow \infty} \cF_{\eps_n} [v_n; \Omega] & = \lim_{n \rightarrow \infty} \lambda_n(\Omega) \ge \lambda(\Omega) \ge \lambda_{ac}(\Omega) = \int_\Omega \frac{d \lambda_{ac}}{d \xi} d \xi  \nonumber \\
& \ge m_d \cH^{d-1}(\Omega \cap \partial^* E) = m_d \mbox{Per}_\Omega (E), \nonumber
\end{align}
which gives \eqref{LimInf} (see \eqref{LimitFunctional} and \eqref{LimitV}).  In the remainder of the proof we show \eqref{C}.

To this end we first note that by the Besicovitch Derivation Theorem (Theorem 1.155 in \cite{fonseca2007modern}), for $\cH^{d-1}$-a.e. $x_0 \in \Omega \cap \partial^* E$
\beq
\label{dLambdadH1}
\infty > \frac{d \lambda_{ac}}{d \cH^{d-1}}(x_0)  = \lim_{r \rightarrow 0^+} \frac{\lambda(\Q_\nu(x_0, r))}{\cH^{d-1}(\Q_\nu(x_0,r) \cap \partial^* E)} =   \lim_{r \rightarrow 0^+} \frac{\lambda(\Q_\nu(x_0, r))}{r^{d-1}},
\eeq
\beq
\label{dZetadH1}
\infty > \frac{d \zeta_{ac}}{d \cH^{d-1}}(x_0)  = \lim_{r \rightarrow 0^+} \frac{\zeta(\Q_\nu(x_0, r))}{\cH^{d-1}(\Q_\nu(x_0,r) \cap \partial^* E)} =   \lim_{r \rightarrow 0^+} \frac{\zeta(\Q_\nu(x_0, r))}{r^{d-1}},
\eeq
where $\nu$ denotes the outward normal vector to $E$ at $x_0$. 
Fix $x_0 \in \Omega \cap \partial^* E$ for which  \eqref{dLambdadH1} and \eqref{dZetadH1} hold.  Then there exists $\bar{r} > 0$ such that
\beq
\frac{\zeta(\Q_\nu(x_0, r))}{r^{d-1}} \le \frac{d \zeta_{ac}}{d \cH^{d-1}} (x_0) + 1 =: M \nonumber
\eeq
for all $0 < r \le \bar{r}$.  Let $0 < r_0 \le \bar{r}$ be such that $\zeta(\partial Q_\nu(x_0, r_0)) = \mu(\partial Q_\nu(x_0, r_0)) = 0$.  Then by Corollary 1.204 in \cite{fonseca2007modern},
\beq
\lim_{n \rightarrow \infty} \frac{\cI_{\eps_n}[v_n; Q_\nu(x_0, r_0)]}{r_0^{d-1}} = \frac{ \zeta(Q_\nu(x_0, r_0))}{r_0^{d-1}} \le M \nonumber
\eeq
and so
\beq
\cI_{\eps_n}[v_n; Q_\nu(x_0, r_0)] \le (M+1) \, r_0^{d-1} \nonumber
\eeq
for all $n \ge n_0 = n_0(r_0)$.  Let $\m \in \mathbb{N}$.  By Proposition \ref{LayerProp} with $K := (M+1) r_0^{d+1}$, for each $n \ge n_0$ there exists $i_n \in \{1, \dots, \m\}$ such that
\beq
\cI_{\eps_n} [v_n; L_{i_n}] \le \frac{(M+1) \, r_0^{d-1}}{\m}, \nonumber
\eeq
where $L_{i_n} := Q_\nu\left(x_0,\frac{r_0}{2} \left( 1 + \frac{i_n}{\m} \right)\right) \backslash Q_\nu \left(x_0,\frac{r_0}{2} \left( 1 + \frac{i_n-1}{\m} \right)\right)$.

Since $i_n \in \{1, \dots, \m\}$ for all $n \ge n_0$, there exists $i^{(1)} \in \{1, \dots, \m\}$ such that $i^{(1)} = i_n$ for infinitely many $n$, say $n_l^{(1)}, l \in \mathbb{N}$.
Let $L^{(1)} := L_{i^{(1)}}$.  Then
\beq
\label{D}
\cI_{\eps_{n_l^{(1)}}} [v_{n_l^{(1)}}; L^{(1)}] \le \frac{ (M+1) \, r_0^{d-1}}{\m} \nonumber
\eeq
for all $n_l^{(1)}$, $l \in \mathbb{N}$.  
We proceed by induction. Let 
$$ \alpha_k:=\max\left\{\frac{4\m(\m+i-1)}{(4\m-1)(\m+i)}:\,i=1,\ldots k\right\}=1-\frac{1}{4k-1}.$$ 
Fix $\beta_k\in (\alpha_k,1)$  and for $j\in\mathbb{N}$ choose 
\beq
\label{F} 
r_j \in \left( \frac{r_{j-1}}{2} \left(1+ \frac{4i^{(j)} -3}{4\m-1} \right), \beta_k\frac{r_{j-1}}{2} \left(1 + \frac{i^{(j)}}{\m} \right) \right)
\eeq
such that $\zeta(\partial Q_\nu(x_0, r_{j})) = \mu(\partial Q_\nu(x_0, r_{j})) = 0$.  Note that $r_j \rightarrow 0^+$ since $r_j<\beta_k r_{j-1}$ for all $j$.
Thus, for every $j$ we find a subsequence $\{v_{n_l^{(j)}}\}_{l \in \mathbb{N}} \subset \{v_{n_l^{(j-1)}} \}_{l \in \mathbb{N}}$ and a layer 
\beq
\label{LjDef}
L^{(j+1)} := Q_\nu\left(x_0,\frac{r_{j}}{2} \left( 1 + \frac{i^{(j+1)}}{\m} \right)\right) \backslash Q_\nu\left(x_0,\frac{r_{j}}{2} \left( 1 + \frac{i^{(j+1)}-1}{\m} \right)\right)
\eeq
such that
\beq
\label{E}
\cI_{\eps_{n_l^{(j+1)}}} [v_{n_l^{(j+1)}}; L^{(j+1)}] \le \frac{ (M+1) \, r_{j}^{d-1}}{\m}
\eeq
for all $l \in \mathbb{N}$.

By \eqref{dLambdadH1} and Corollary 1.204 in \cite{fonseca2007modern},
\beq
\frac{d \lambda_{ac}}{d \cH^{d-1}} (x_0) = \lim_{j \rightarrow \infty} \frac{\lambda(Q_\nu(x_0, r_j))}{r_j^{d-1}} = \lim_{j \rightarrow \infty} \lim_{n \rightarrow \infty} \frac{\cF_{\eps_n}[v_n; Q_\nu(x_0, r_j)]}{r_j^{d-1}}, \nonumber
\eeq
and by Theorem 3.59 from {\cite{ambrosio2000functions} (see also \eqref{MT2})
\beq
 \lim_{j \rightarrow \infty} \lim_{n \rightarrow \infty} \frac{1}{r_j^{d}} \int_{Q_\nu(x_0, r_j)}  |v_{n} - \tilde{v}_0|^2 dx   =  \lim_{j \rightarrow \infty} \frac{1}{r_j^{d}} \int_{Q_\nu(x_0, r_j)}  |v - \tilde{v}_0|^2 dx = 0, \nonumber
 \label{vDiffu0}
\eeq
where $\tilde{v}_0(x):=v_0(x-x_0)$ with $v_0$ introduced in Proposition \ref{glProp}.
Further, by \eqref{E} and using the fact that for $j\in\mathbb{N}$, $\eps_{n_l^{(j)}} \to 0$ as $l\to\infty$, we can use a diagonal argument to find $\eps^{(j)} \in \{ \eps_{n_l^{(j)}} \}_{l \in \mathbb{N}}$ and $\tilde{v}_j \in \{v_{n_l^{(j)}} \}_{l \in \mathbb{N}}$ such that $\eps^{(j)}/r_{j} \rightarrow 0$,
\beq
\label{DlambdaDiag}
\frac{d \lambda_{ac}}{d \cH^{d-1}} (x_0) =  \lim_{j \rightarrow \infty}  \frac{\cF_{\eps^{(j)}} [\tilde{v}_j; Q_\nu(x_0, r_j)]}{r_j^{d-1}},
\eeq
\beq
\label{VjConv}
 \lim_{j \rightarrow \infty}  \frac{1}{r_j^{d}} \int_{Q_\nu(x_0, r_j)}  |\tilde{v}_{j} - \tilde{v}_0|^2 dx = 0,\text{\qquad and}
\eeq
\beq
\label{ILayerBound}
\cI_{\eps^{(j)}} [\tilde{v}_j; L^{(j)}] \le \frac{ (M+1) \, r_{j}^{d-1}}{\m}.
\eeq
Define
\beq
w_j(y) := \tilde{v}_j(x_0 + r_j y), \quad y \in Q_\nu(0,1), \nonumber
\eeq
and (see Proposition \ref{glProp})
\beq
\tilde{L}_k := Q_\nu(0,1) \backslash Q_\nu(0, 1- 1/(4\m)). \nonumber
\eeq
Since $L^{(j)} \supseteq Q_\nu(x_0, r_j) \backslash Q_\nu(x_0, r_j(1 - 1/(4\m)))=x_0+r_j\tilde{L}_k$ by \eqref{LjDef} and \eqref{F}, by \eqref{ILayerBound} we have
\beq
\cI_{\eps^{(j)}/r_j}[w_j; \tilde{L}_k] =  \frac{1}{r_j^{d-1}} \cI_{\eps^{(j)}}[\tilde{v}_j; x_0 + r_j \tilde{L}_k]  \le \frac{1}{r_j^{d-1}} \cI_{\eps^{(j)}}[\tilde{v}_j; L^{(j)}] \le  \frac{r_{j-1}^{d-1}}{r_j^{d-1}} \frac{(M+1) }{\m} \le \frac{(M+1) \, 2^{d-1}}{\m}, \nonumber
\eeq
where we also used $r_j > \frac{r_{j-1}}{2}$.
Moreover \eqref{DlambdaDiag} and \eqref{VjConv} become
\beq
\frac{d \lambda_{ac}}{d \cH^{d-1}} (x_0) =  \lim_{j \rightarrow \infty}  \cF_{\eps^{(j)}/r_j} [w_j; Q_\nu(0, 1)], \nonumber
\eeq
and
\beq
 \lim_{j \rightarrow \infty}  \int_{Q_\nu(0, 1)}  |w_j - v_0|^2 dy = 0. \nonumber
\eeq
We can apply Proposition \ref{glProp} to obtain
\beq
\frac{d \lambda_{ac}}{d \cH^{d-1}} (x_0) \ge m_d - \frac{C}{\m}. \nonumber
\eeq
Letting $\m \rightarrow \infty$ completes the proof.
\section{Proof of the Limsup Inequality}\label{sec:limsup}
We now turn to the proof of \eqref{Limsup}, where we follow closely the argument in \cite{chermisi2010singular}.
\vspace{0.2cm}

\noindent \emph{Step 1.} Assume first that the target function $v$ has a flat interface orthogonal to a given direction $\nu \in \mathbb{S}^{d-1}$, and that $\Omega$ has a Lipschitz boundary that meets this interface orthogonally.  More precisely, without loss of generality (under suitable rigid transformations of the coordinate system), we assume that $v\in BV(\Omega;\{\pm 1\})$ is of the simple form
\begin{align}
v(x) := \left\{\begin{array}{ll}
-1  &\mbox{ if  }  x_{\N} < 0, \\
1 & \mbox{ if } x_{\N} > 0,
\end{array}\right. \nonumber
\end{align}
where we use the notation $x_d:=x\cdot e_{\d}=x\cdot\nu$,
and  that the normal to $\partial \Omega$ is orthogonal to $e_{\N}$ for all $x \in \partial \Omega$ with $|x_{\N}|$ small enough.
Let $\rho > 0$.  By definition of $m_{\d}$ (see \eqref{eq:md} and the remark after), there exist $\eps_0 > 0$ and $w \in \cA_\nu$ such that
\beq
\int_{Q} \left( \frac{1}{\eps_0} \W(-\eps_0^2 \Delta w + w) - \eps_0 q \left| \nabla w\right|^2 + (1-2q) \eps_0^3 |\Delta w|^2 + (1-q) \eps_0^5|\nabla \Delta w|^2 \right) \,dx < m_{\d} + \rho.
\label{eq:Defw}
\eeq
Define
\begin{align}
w_n(x) := \left\{\begin{array}{ll}
-1  &\mbox{ if  }  x_{\N} < -\frac{\eps_n}{2 \eps_0}, \\
w \left(\frac{\eps_0 x}{\eps_n}\right) & \mbox{ if } |x_{\N}| \le \frac{\eps_n}{2\eps_0}, \\
1 & \mbox{ if } x_{\N} > \frac{\eps_n}{2 \eps_0}.
\end{array}\right. \nonumber
\label{eq:wn}
\end{align}
Note that, for $n$ large enough, $w_n\in W^{3,2}(\Omega)$.
Moreover, we claim that $w_n\rightarrow v$ in $L^2(\Omega)$.  Indeed,
\beq
\|w_n-v\|_{L^2(\Omega)}=\|w_n-v\|_{L^2(\{x\in\Omega:\,|x_{\N}|<\frac{\eps_n}{2\eps_0}\})}\leq \|w_n\|_{L^2(\{x\in\Omega:\,|x_{\N}|<\frac{\eps_n}{2\eps_0}\})}+\|v\|_{L^2(\{x\in\Omega:\,|x_{\N}|<\frac{\eps_n}{2\eps_0}\})},  \nonumber
\eeq
where for $n$ sufficiently large
\beq
\|v\|_{L^2(\{x\in\Omega:\,|x_{\N}|<\frac{\eps_n}{2\eps_0}\})}=\left|\left\{x\in\Omega:\,|x_{\N}|<\frac{\eps_n}{2\eps_0}\right\}\right|\rightarrow 0  \mbox{ as } n \rightarrow \infty. \nonumber
\eeq
Further, setting $\Omega':=\{x'\in\R^{\N-1}:\,(x',0)\in\Omega\}$, we have for sufficiently large $n$, that $\{x \in \Omega: |x_d| \le \eps_n/(2\eps_0)\} = \Omega' \times [-\eps_n/(2\eps_0), \eps_n/(2\eps_0)]$. Hence, applying the change of variables $t:=\frac{\eps_0 x_{\N}}{\eps_n}$ yields
\beq
\label{wnNorm}
 \|w_n\|^2_{L^2(\{x\in\Omega:\,|x_{\N}|<\frac{\eps_n}{2\eps_0}\})}  = \int_{\left\{x \in \Omega: |x_d| < \frac{\eps_n}{2\eps_0} \right\}} \left| w \left(\frac{\eps_0 x}{\eps_n} \right) \right|^2 dx = \frac{\eps_n}{\eps_0} \int_{-1/2}^{1/2} \int_{\Omega'} \left| w \left(\frac{\eps_0 x'}{\eps_n}, t \right) \right|^2 dx'  d t.
\eeq
Since $w$ is periodic in the first $\d-1$ arguments, applying   Fubini's Theorem  and the Riemann-Lebesgue Lemma (see for example Lemma 2.85 in \cite{fonseca2007modern}) to  $\int_{-1/2}^{1/2} \left| w\left(\frac{\eps_0 x'}{\eps_n}, t\right) \right|^2 dt \in L^1_{\text{loc}}(\R^{d-1})$ gives
\beq
\lim_{n \rightarrow \infty}  \int_{\Omega'} \int_{-1/2}^{1/2} \left| w \left(\frac{\eps_0 x'}{\eps_n}, t \right) \right|^2   d t \ dx' = \int_{\Omega'}  \int_{Q'}  \int_{-1/2}^{1/2}  |w(y, t)|^2  \ dt \ dy  \ dx' = \cL^{d-1}(\Omega') ||w||^2_{L^2(Q)}. \nonumber
\eeq
It then follows from \eqref{wnNorm} that 
\beq
 \|w_n\|^2_{L^2(\{x\in\Omega:\,|x_{\N}|<\frac{\eps_n}{2\eps_0}\})}   \le \frac{C \eps_n}{\eps_0} ||w||_{L^2(Q)}^2 \rightarrow 0, \mbox{ as } n \rightarrow \infty. \nonumber
\eeq
This concludes the proof that $w_n \rightarrow v$ in $L^2(\Omega)$. 

Since $w_n = \pm 1$ on $\{x \in \Omega: |x_d| \ge \frac{\eps_n}{2\eps_0}\}$, the contribution to the energy only comes from the interfacial region $\{x \in \Omega: |x_d| \le \frac{\eps_n}{2\eps_0}\}$, where we have
\beq
-\eps_n^2 \Delta w_n(x) + w_n(x) = -\eps_0^2 \Delta w\left(\frac{\eps_0 x}{\eps_n}\right) + w\left(\frac{\eps_0 x}{\eps_n}\right).  \nonumber
\eeq
Setting, as before, $t := \frac{\eps_0 x_d}{\eps_n}$ we have for $n$ sufficiently large
\begin{align}
 \mathcal{F}_{\eps_n}[w_n; \Omega]&= \int_{\{x\in\Omega:\,|x_{\N}|<\frac{\eps_n}{2\eps_0}\}}\left\{\frac{1}{\eps_n}W\left(-\eps_0^2\Delta w+w\right)-\frac{\eps_0^2}{\eps_n}q\left|\nabla w\right|^2+  (1-2q) \frac{\eps_0^4}{\eps_n}|\Delta w|^2 + \right. \nonumber \\
&+ \left. (1-q) \frac{\eps_0^6}{\eps_n}|\nabla\Delta w|^2
\right\}\left(\frac{\eps_0 x}{\eps_n}\right)\,dx \nonumber \\
&= \int_{\Omega'}\int_{-\frac{1}{2}}^{\frac{1}{2}} \Bigg\{\frac{1}{\eps_0}W\left(-\eps_0^2\Delta w+w\right) -q\eps_0\left|\nabla w\right|^2 + \nonumber \\
&+ (1-2q) \eps_0^3|\Delta w|^2 + (1-q) \eps_0^5|\nabla\Delta w|^2 \Bigg\} \left(\frac{\eps_0x'}{\eps_n},t\right)\,dt  \ dx'. \nonumber
\end{align}
Since $w$ is periodic in the first $\d-1$ arguments, also the functions
\begin{eqnarray*}
&& x'\mapsto\int_{-\frac{1}{2}}^{\frac{1}{2}}W(-\eps_0^2\Delta w+w)(x',t)\,dt,\quad x'\mapsto\int_{-\frac{1}{2}}^{\frac{1}{2}}\left|\nabla w\right|^2(x',t)\,dt,\\
&&x'\mapsto\int_{\frac{1}{2}}^{\frac{1}{2}}|\Delta w|^2(x',t)\,dt,\quad\mbox{and}\quad x'\mapsto\int_{-\frac{1}{2}}^{\frac{1}{2}}|\nabla\Delta w|^2(x',t)\,dt \nonumber
\end{eqnarray*}
are periodic and locally in $L^1$, where for the integral involving $W$ we used the quadratic growth assumption from Hypotheses \ref{assume1}. Thus, by the Riemann-Lebesgue Lemma and the choice of $w$ (see \eqref{eq:Defw}),
\begin{align}
 \lim_{n\rightarrow\infty} \mathcal{F}_{\eps_n}[w_n; \Omega]&=\mathcal{L}^{\N-1}(\Omega')\int_Q\Big\{\frac{1}{\eps_0}W(-\eps_0^2\Delta w+w)-q\eps_0\left|\nabla w\right|^2+ (1- 2q) \eps_0^3|\Delta w|^2 \nonumber \\
& + (1-q) \eps_0^5|\nabla\Delta w|^2\Big\}\,dx  \leq (m_{\d}+\rho)\mbox{Per}_{\Omega}(\{v=1\}),
\label{eq:imppart}
\end{align}
and the limsup inequality follows since $\rho>0$ is arbitrarily small.
\vspace{0.2cm}

\begin{figure}
\begin{center}
\includegraphics[height=2in]{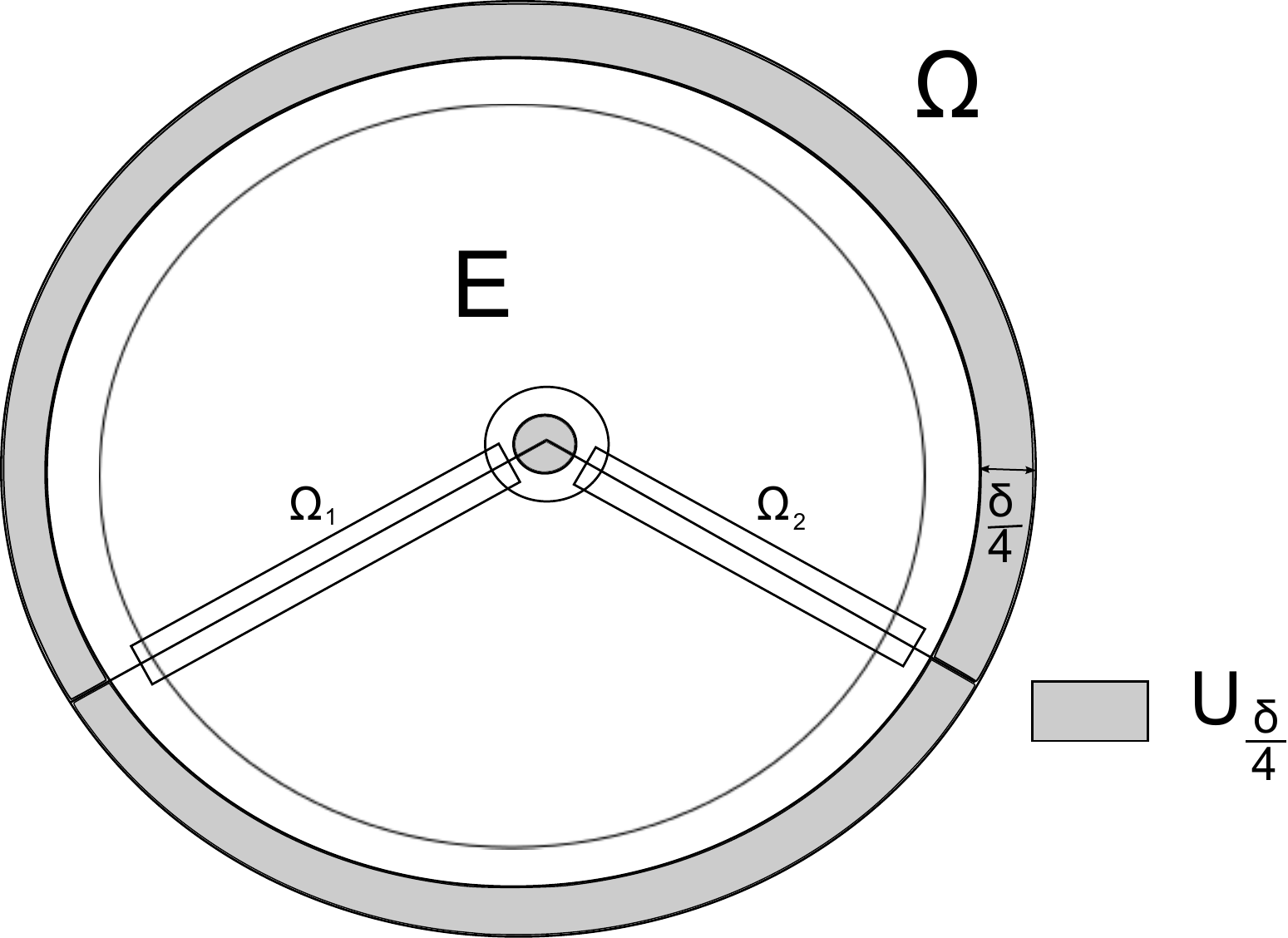}
\vskip 0.1in
\caption{Construction in Step 2.}
\label{experiment2}
\end{center}
\end{figure}

\noindent \emph{Step 2.} Consider now the case in which
\beq
v = \chi_{E} - \chi_{\Omega \backslash E}, \nonumber
\eeq
where $\mbox{Per}_\Omega(E) < \infty$ and $E$ has the form $E=P\cap\Omega$ with $P$ a polyhedron, i.e., there is $L \in \mathbb{N}$ such that $\partial P=H_1\cup H_2\cup\dots\cup H_L\cup F$ with pairwise disjoint relatively open convex polyhedra $H_i$ of dimension $\N-1$, $H_i\subset\{x\in\R^{\N}:\,(x-x_i)\cdot\nu_i=0\}$ for some $x_i\in\R^{\N}$ and $\nu_i\in S^{\N-1}$, $i = 1, \dots, L$, and $F$ is the union of a finite number of convex polyhedra of dimension $\N-2$.  Finally, we assume that $E$ meets the boundary of $\Omega$ transversally, more precisely
\beq
\label{trans}
\partial \Omega \cap \partial P \mbox{ is the union of a finite number of $C^1$ manifolds of dimension $d-2$}.
\eeq
 We extend $v$ to $\R^{\N}$ by setting
\[v(x):=\chi_P(x)-\chi_{\R^{\N}\setminus P}(x), \]
and define 
\beq
\label{recovMoll}
\varphi_n:=v\ast\Psi_{\eps_n} 
\eeq
 with mollifiers $\Psi_{\eps_n}$ (see \eqref{mol}).
For fixed (small) $0<\delta<1$ set
\beq
U_\delta:=\left\{x\in\Omega:\,\mbox{dist}(x, \partial\Omega \cup F)\leq\delta\right\} \nonumber
\eeq
and let $H_i'$ be relatively open subsets of $H_i$ with a $d-2$ dimensional $C^\infty$ boundary such that
\beq
\left\{x\in H_i\cap\Omega:\,\mbox{dist}(x,\partial\Omega\cup F) \ge \frac{\delta}{2}\right\}\subset H_i'\subset\overline{H_i'}\subset H_i\cap\Omega \nonumber
\eeq
and $\overline{H_i'} \cap U_{\frac{\delta}{4}} = \emptyset$.
Fix $0<\eta<\delta/2$, and set for every $i=1,2, \dots, L$,
\beq
\Omega_i:=\left\{x+t\nu_i:\,x\in H_i',\,|t|<\eta\right\}.  \nonumber
\eeq
Taking $\eta$ sufficiently small we may assume, without loss of generality, that $\Omega_1,\dots,\Omega_L$ are pairwise disjoint and
\beq
\label{Omi}
\overline{\Omega_i} \cap U_{\frac{\delta}{4}} = \emptyset.
\eeq
We apply {\em Step 1} to every $\Omega_i$ to obtain a sequence $\{w_n^i\} \subset W^{3,2}(\Omega_i)$ such that $w_n^i\rightarrow v$ in $L^2(\Omega_i)$, and $\lim_{n\rightarrow\infty}\mathcal{F}_{\eps_n}[w_n^i;\Omega_i]\leq(m_{\N}+\rho)\mathcal{H}^{\N-1}(H_i\cap\Omega_i)$.  For every $\delta > 0$ choose cut-off functions $\eta_\delta\in C_c^\infty(\R^{\N};[0,1])$ such that
\begin{align}
\label{etaDerivBounds}
\eta_\delta=0\mbox{\ in\ }U_\delta,\quad \eta_\delta=1\mbox{\ in\ }\R^{\N}\setminus U_{2\delta}, \quad\|\nabla^k\eta_\delta\|_{L_\infty(\R^{\N})}\leq C/\delta^k\mbox{\ for\ }k=1,2,3.
\end{align}
Define $V_n$ by
\begin{align}
 \label{VNDef}
V_n:=\begin{cases}
      \eta_\delta w_n^i+(1-\eta_\delta)\varphi_n&\mbox{in\ }\overline{\Omega_i}, \quad  i = 1, \dots, L,\\
\eta_{\frac{\delta}{8}} \varphi_n&\mbox{in\ }A:=\Omega\setminus(\overline{\Omega_1}\cup\dots \cup \overline{\Omega_L}).
     \end{cases}
\end{align}
We claim that $V_n \in W^{3,2}(\Omega)$ and satisfies Neumann boundary conditions on $\partial \Omega$. Indeed, considering $V_n$ in the neighborhood of $\partial A$, we observe that
by construction of $w_n^i$ in \emph{Step 1}
\beq
w_n^i(x) = v(x) \quad \mbox{ for } x \in \overline{\Omega_i} \mbox{ and } \mbox{dist}(x, H_i) \ge \frac{\eps_n}{2\eps_0}. \nonumber
\eeq
Hence, from \eqref{recovMoll}, for sufficiently large $n$ we have $w_n^i = \varphi_n$  in a neighborhood of $\{x \in \partial \Omega_i: \mbox{dist}(x,H_i) = \eta \}$ (the part of $\partial \Omega_i$ parallel to $H_i$),
and by \eqref{Omi} in that region both  $\eta_\delta w_n^i+(1-\eta_\delta)\varphi_n$ and $\eta_{\frac{\delta}{8}} \varphi_n$ are equal to $\varphi_n$.
In addition, $\{x \in \partial \Omega_i: \mbox{dist}(x,H_i) < \eta \}$ (the part of $\partial \Omega_i$ orthogonal to $H_i$) is contained in $U_{\delta} \backslash U_{\delta/4}$ and   both  $\eta_\delta w_n^i+(1-\eta_\delta)\varphi_n$ and $\eta_{\frac{\delta}{8}} \varphi_n$ are equal to $\varphi_n$ also in that region.  Finally, $V_n$ is identically zero in a neighborhood of $U_{\frac{\delta}{8}}$ so the Neumann boundary conditions are satisfied.

Furthermore, $\lim_{n \rightarrow \infty} ||V_n - v||_{L^2(\Omega)} \le C\delta$, since $w_n^i\rightarrow v$ in $L^2(\Omega_i)$ and $\varphi_n \rightarrow v$ in $L^2(\Omega \backslash U_{\frac{\delta}{8}})$. 
It remains to estimate the energies.  By \eqref{recovMoll}, $V_n$ is possibly different from $\pm 1$ only on $U_{\frac{\delta}{4}}$ and on
\beq
R_n:=\left\{x\in\Omega:\,\mbox{dist}(x,\partial P)\leq\max\{\eps_n/(2\eps_0),\,\eps_n\}\right\}. \nonumber
\eeq
Using the notation from \eqref{VNDef}, $V_n = \eta_{\frac{\delta}{8}} \varphi_n$ on $U_{\frac{\delta}{4}}$, $V_n = \varphi_n$ on $A \backslash U_{\frac{\delta}{4}}$, $A\cap R_n\subset U_\delta$ and $\mathcal{H}^{\N-1}(\partial P\cap U_\delta)\leq C\delta$.  Thus, for $n$ sufficiently large,
\begin{align}
\mathcal{F}_{\eps_n}[V_n;A] & \leq \left| \cF_{\eps_n}[\eta_{\frac{\delta}{8}} \varphi_n; U_{\frac{\delta}{4}}] \right|+ \int_{A\cap R_n}\Big(\frac{1}{\eps_n}W(-\eps_n^2 \Delta \varphi_n + \varphi_n) + \eps_n |q| |\nabla \varphi_n|^2+ (1-2q) \eps_n^3|\Delta \varphi_n|^2 \nonumber \\
& + (1-q) \eps_n^5|\nabla\Delta \varphi_n|^2\Big) dx \leq C\delta,  \nonumber
\end{align}
where we also used \eqref{phiBounds} and  \eqref{etaDerivBounds} to bound the derivatives of $\varphi_n$ and $\eta_{\frac{\delta}{8}}$, respectively.  Next we estimate the energy in $\Omega_i$.
In $\Omega_i \cap U_\delta$, $V_n = \varphi_n$ and using  \eqref{phiBounds} yields
\beq
\cF_{\eps_n}[V_n; \Omega_i\cap U_\delta] = \cF_{\eps_n}[V_n; \Omega_i\cap U_\delta \cap R_n] \leq C\delta.
 \label{eq:K1}
\eeq
To obtain estimates inside $T := \Omega_i \cap (U_{2\delta} \backslash U_\delta)$ we first observe that
\begin{align}
\partial_{x_i} V_n = w_n^i \partial_{x_i} \eta_{\delta}  + \eta_{\delta} \partial_{x_i} w_n^i -  \varphi_n \partial_{x_i} \eta_{\delta}  + (1-\eta_{\delta}) \partial_{x_i} \varphi_n, \nonumber
\end{align}
and arguing as in \eqref{eq:imppart},
\beq
\label{wnDerivBounds}
\lim_{n \rightarrow \infty} \eps_n^{2k-1} || \nabla^k w_n^i ||_{L^2(\Omega_i \cap U_{2\delta})}^2 \le C(\rho) \, \cH^{d-1}(H_i \cap U_{2\delta})  \le C(\rho) \, \delta \mbox{ for } k =0, \dots 3, 
\eeq
where we also used the fact that $w \in W_{loc}^{3,\infty}(\R^d)$.
Combined with the bounds on $\varphi_n$ from \eqref{phiBounds}, it follows that,
\begin{align}
\int_T \eps_n |\nabla V_n|^2 dx & = \int_{T \cap R_n} \eps_n |\nabla V_n|^2 dx \le C(\rho) \Big( \frac{\eps_n}{\delta^2} ||w_n^i||_{L^2(T)}^2 + \eps_n ||\nabla w_n^i||_{L^2(T)}^2 + \frac{\eps_n}{\delta^2} ||\varphi_n||_{L^2(T)}^2 + \nonumber \\
& +   \eps_n ||\nabla \varphi_n||_{L^2(T)}^2 \Big) \le C(\rho) \left( \delta + \frac{\eps_n}{\delta^2} \right). \nonumber
\end{align}
Analogous calculations for the higher derivatives of $V_n$, yield the bound
\beq
 \cF_{\eps_n}[V_n;\Omega_i\cap(U_{2\delta}\setminus U_\delta)]\leq C(\rho) \, \delta
 \label{eq:K2} 
\eeq
for $n$ sufficiently large.
Next, by \eqref{VNDef}, \eqref{wnDerivBounds} and \eqref{phiBounds}, we have
\beq
\lim_{n\rightarrow\infty} \int_{\Omega_i\cap U_{2\delta}}\eps_n|\nabla V_n|^2\,dx\leq C(\rho) \, \delta, \nonumber
\eeq
and hence
\begin{align}
&\int_{\Omega_i\backslash U_{2\delta}}\left\{\frac{1}{\eps_n}W(V_n)-\eps_nq|\nabla V_n|^2+ (1-2q) \eps_n^3|\Delta V_n|^2+ (1 - q) \eps_n^5|\nabla\Delta V_n|^2\right\}\,dx\nonumber\\
&\leq\int_{\Omega_i}\left\{\frac{1}{\eps_n}W(V_n)-\eps_nq|\nabla V_n|^2+ (1-2q) \eps_n^3|\Delta V_n|^2+ (1-q) \eps_n^5|\nabla\Delta V_n|^2\right\}\,dx+C(\rho)\, \delta. \nonumber \\
 \label{eq:K3}
\end{align}
Combining \eqref{eq:imppart}, \eqref{eq:K1}, \eqref{eq:K2}, and \eqref{eq:K3}, we obtain for $\delta$ sufficiently small a sequence $V_n \in W^{3,2}(\Omega)$,  with Neumann boundary conditions on $\partial \Omega$, satisfying
\beq
\lim_{n \rightarrow \infty} ||V_n - v||_{L^2(\Omega)} \le \rho \nonumber
\eeq
and
\begin{align}
 \limsup_{n\rightarrow\infty} \mathcal{F}_{\eps_n}[V_n]& = \limsup_{n\rightarrow\infty} \mathcal{F}_{\eps_n}[V_n;\Omega] \leq \limsup_{n\rightarrow\infty}\sum_{i=1}^L \mathcal{F}_{\eps_n}[V_n;\Omega_i]+\limsup_{n\rightarrow\infty}\mathcal{F}_{\eps_n}[V_n;A] \nonumber\\
&\leq (m_{\d}+\rho)\sum_{i=1}^L\mathcal{H}^{\N-1}(\Omega_i\cap H_i)+C(\rho) \, \delta \nonumber \\
&\leq (m_{\d}+\rho)\mathcal{H}^{\N-1}(\Omega\cap\partial P)+ \rho, \nonumber
\end{align}
and the Limsup Inequality \eqref{LimInf} follows by a standard diagonalizing argument.

\vspace{0.2cm}

\noindent \emph{Step 3.} Lastly we consider the case in which the target function is
\beq
v = \chi_{E} - \chi_{\Omega \backslash E}, \nonumber
\eeq
where $E$ is an arbitrary set of finite perimeter in $\Omega$.  Since $\Omega$ is bounded and has $C^2$ boundary, we can approximate $E$ with smooth sets (see Remark 3.43 in \cite{ambrosio2000functions}) and then with polyhedral sets.  In particular, we may find sets $E_k \subset \Omega$ of the form $E_k = P_k \cap \Omega$, where $P_k$ are polyhedral sets satisfying \eqref{trans} such that $\cH^{d-1}(\partial E_k \cap \partial \Omega) = 0$, $\chi_{E_k} \rightarrow \chi_{E}$ in $L^2(\Omega)$, and $\mbox{Per}_\Omega(E_k) \rightarrow \mbox{Per}_{\Omega}(E)$ as $k \rightarrow +\infty$.  We apply \emph{Step 2} to each function $v_k := \chi_{E_k} - \chi_{\Omega \backslash E_k}$ to find a sequence
\beq
V_n^k \rightarrow v_k \nonumber
\eeq
satisfying
\beq
 \limsup_{n\rightarrow\infty} \cF_{\eps_n}[V_n^k;\Omega] \le m_{\d}\mathcal{H}^{\N-1}(E_k\cap\partial P_k) \nonumber
\eeq
and
\beq
\limsup_{k \rightarrow \infty} \limsup_{n\rightarrow\infty} \cF_{\eps_n}[V_n^k] \le \limsup_{k \rightarrow \infty} \left( m_{\d} \mathcal{H}^{\N-1}(E_k\cap\partial P_k) \right) = m_{\d} \mbox{Per}_{\Omega}(E). \nonumber
\eeq
The general  result now follows by a diagonalizing argument.

\section*{Appendix}

We derive the energy functional \eqref{mainFunct} from  \eqref{mainEnergy}. To eliminate the dependence on $h$ we assume that $\phi$ and $h$ satisfy the Euler-Lagrange equation
\beq
\label{ELh}
\frac{\delta \mathcal{E}}{\delta h}(\phi,h) = 0.
\eeq
After changing variables, $x := \bar{x}/L$, $u(x) := \phi(\bar{x})$ in \eqref{mainEnergy} we have
\beq
\label{mainEnergy2}
\frac{1}{L^d}\mathcal{E}[u, h] = \int_\Omega \left(f(u) + \frac{b}{2L^2}  |\nabla u|^2 + \frac{\sigma}{2 L^2} |\nabla h|^2 + \frac{\kappa}{2 L^4} [\Delta h]^2 +\frac{ \Lambda}{L^2} u \Delta h \right) \,dx,
\eeq
where $\Omega := \{ x/L: x \in D\}$.
Assuming natural boundary conditions, the Euler-Lagrange equation \eqref{ELh} takes the form
\begin{align}
\label{ELEqn}
\begin{cases}
\Delta \left(\frac{\kappa}{L^4}  \Delta h - \frac{\sigma}{L^2} h +\frac{\Lambda}{L^2} u \right) = 0 & \text{ in } \Omega, \\
\frac{\partial h}{\partial \n} = 0, \frac{\partial \Delta h}{\partial \n} = 0, \frac{\partial u}{\partial \n} = 0, \frac{\partial \Delta u}{\partial \n} = 0 & \text{ on } \partial \Omega.
     \end{cases}
\end{align}
Consider the Fourier Series expansions of $h$ and $u$,
\begin{align}
h = \sum_{i=0}^{\infty} h_i \psi_i, \quad
u = \sum_{i=0}^{\infty} u_i \psi_i,  \nonumber
\end{align}
where $\psi_i$ are the eigenfunctions of $-\Delta$ on $H^1(\Omega)$ with Neumann boundary conditions.  Denote the corresponding nonnegative eigenvalues by $\lambda_i^2$.
Then, since $\psi_0 = const$ (due to Neumann boundary conditions), we have
\beq
\Delta h = -\sum_{i=1}^\infty  \lambda_i^2 h_i \psi_i,
\mbox{\quad and\quad }
\Delta^2 h = \sum_{i=1}^\infty \lambda_i^4 h_i \psi_i, \nonumber
\eeq
and thus by \eqref{ELEqn} 
\beq
\sum_{i=1}^\infty \lambda_i^2 \left(\frac{\kappa}{L^2} \lambda_i^2 h_i + \sigma h_i - \Lambda u_i \right) \psi_i = 0. \nonumber
\eeq
Taking the $L^2$ inner product with $\psi_j$, and noting that $\left< \psi_i, \psi_j \right>_{L^2(\Omega)} = \delta_{ij}$, we obtain
\beq
\lambda_j^2 \left(\frac{\kappa}{L^2} \lambda_j^2 h_j  + \sigma h_j  - \Lambda u_j \right) = 0 \mbox{ for } j=1, \dots, \infty. \nonumber
\eeq
Solving for $h_j$ yields
\beq
h_j = \frac{\Lambda u_j}{\sigma + (\kappa/L^2) \lambda_j^2} \mbox{ for } j =1, \dots, \infty, \nonumber
\eeq
and
\beq
h(x) = \sum_{i=0}^\infty h_i \psi_i(x) = const + \sum_{i=1}^\infty h_i \psi_i(x) = const + \sum_{i=1}^\infty \frac{\Lambda u_i \psi_i(x)}{\sigma + (\kappa/L^2) \lambda_i^2}. \nonumber 
\eeq
Using this expansion and $\Delta \psi_i = -\lambda_i^2 \psi_i$ gives
\beq
\label{DHExp}
-\Delta h(x) =  \sum_{i=1}^\infty \frac{\Lambda \lambda_i^2 u_i \psi_i(x)}{\sigma + (\kappa/L^2) \lambda_i^2}.
\eeq
In addition, multiplying \eqref{ELEqn} by $h$ and integrating by parts, we obtain
\beq
\int_\Omega \left( (\kappa/L^2) (\Delta h)^2 + \sigma |\nabla h|^2 + \Lambda u \Delta h \right) \,dx = 0, \nonumber
\eeq
and consequently
\beq
\label{hRed}
\frac{1}{2} \int_\Omega \left( (\kappa/L^2) (\Delta h)^2 + \sigma |\nabla h|^2 \right)\,dx  = - \frac{1}{2} \int_\Omega \Lambda u \Delta h  \,dx .
\eeq
Substituting \eqref{hRed} into \eqref{mainEnergy2} yields
\beq
\label{redF}
\frac{1}{L^d} \mathcal{E}[u, h] = \int_\Omega \left( f(u) + \frac{b}{2L^2}  |\nabla u|^2 + \frac{\Lambda}{2 L^2}  u \Delta h \right)\,dx.
\eeq
To eliminate the dependence on $h$ observe that since $\langle \psi_i, \psi_j \rangle_{L^2(\Omega)} = \delta_{ij}$, \eqref{DHExp} implies that
\beq
\int_\Omega u \Delta h \,dx = -\int_\Omega  \left(\sum_{i=0}^{\infty} u_i \psi_i(x)\right) \left(  \sum_{j=1}^\infty \frac{\Lambda \lambda_j^2 u_j \psi_j(x)}{\sigma + (\kappa/L^2) \lambda_j^2} \right) dx =  -\sum_{i=1}^\infty \frac{\Lambda \lambda_i^2 u_i^2}{\sigma + (\kappa/L^2) \lambda_i^2}. \nonumber
\eeq
Substituting this expression into the energy functional \eqref{redF} yields
\begin{align}
\frac{1}{L^d} \mathcal{E}[u] &= \int_\Omega \left(f(u) + \frac{b}{2 L^2} |\nabla u|^2 \right) \,dx - \frac{1}{2 L^2} \sum_{i=1}^\infty  \frac{\Lambda^2 \lambda_i^2} {\sigma + (\kappa/L^2) \lambda_i^2}   u_i^2\nonumber\\
& = \int_\Omega \left(f(u) + \frac{b}{2 L^2} |\nabla u|^2 \right) \,dx -  \frac{\Lambda^2}{2\kappa} \sum_{i=1}^\infty \left( \frac{(\kappa/L^2) \lambda_i^2 + \sigma  - \sigma} {\sigma + (\kappa/L^2) \lambda_i^2} \right)  u_i^2\nonumber\\
&= \int_\Omega \left(f(u) + \frac{b}{2L^2}  |\nabla u|^2 \right) \,dx - \frac{\Lambda^2}{2\kappa}  \sum_{i=1}^\infty u_i^2 + \frac{\Lambda^2}{2\kappa}  \sum_{i=1}^\infty \left( \frac{ \sigma} {\sigma + (\kappa/L^2) \lambda_i^2} \right)  u_i^2.
\label{fullF1}
\end{align}
At this point one can use a long-wavelength approximation as suggested for example in \cite{Komura:Langmuir:2006} resulting in
an approximation energy
\begin{align}
\label{AppEn}
\frac{1}{L^d} \mathcal{E}_{ap}[u] 
& =  \int_\Omega \left(f(u) + \frac{1}{2 L^2} \left(b - \frac{\Lambda^2}{\sigma} \right) |\nabla u|^2 + \frac{\Lambda^2 \kappa}{2 L^4 \sigma^2} (\Delta u)^2 \right)\,dx,
\end{align}
which was studied in \cite{chermisi2010singular,zeppieri41asymptotic}. 
Returning to the full energy in \eqref{fullF1}, we have
\begin{align}
\frac{1}{L^d} \mathcal{E}[u] &=  \int_\Omega \left( f(u) + \frac{b}{2 L^2}  |\nabla u|^2 - \frac{\Lambda^2}{2\kappa} u^2 \right) \,dx + \frac{\Lambda^2 L^2 \sigma}{2\kappa^2} \sum_{i=1}^\infty  \frac{1} {\frac{L^2 \sigma}{\kappa} +  \lambda_i^2} u_i^2 \nonumber \\
& = \int_\Omega \left( f(u) + \frac{b}{2 L^2} |\nabla u|^2 - \frac{\Lambda^2}{2\kappa} u^2 +  \frac{\Lambda^2 L^2 \sigma}{2\kappa^2} u \left(\frac{L^2 \sigma}{\kappa} - \Delta \right)^{-1} u \right) \,dx \nonumber \\ 
& =  \int_\Omega \left( f(u) - \frac{\Lambda^2}{2\kappa }  u^2 + \frac{b}{2 L^2} |\nabla u|^2 +  \frac{\Lambda^2}{2\kappa} u \left( {\bf{1}} - \frac{\kappa}{L^2 \sigma} \Delta \right)^{-1} u \right) \,dx \nonumber \\
& = \frac{\Lambda^2}{2 \kappa} \int_\Omega \left( \frac{2\kappa}{\Lambda^2} f(u) - u^2 + \frac{ \kappa b}{L^2 \Lambda^2} |\nabla u|^2 + u \left({ \bf{1}} - \frac{\kappa}{L^2 \sigma} \Delta \right)^{-1} u \right) \,dx.  \nonumber
\end{align}
Setting
\begin{align}
\eps :=  \sqrt{\frac{\kappa}{L^2 \sigma}},
\quad q := 1 - \frac{b \sigma}{\Lambda^2},
\quad
W(u) := \frac{2\kappa}{\Lambda^2} f(u),
\mbox{\quad and\ }
\cFv_\eps := \frac{1}{\eps}\frac{2\kappa}{\Lambda^2 L^d} \mathcal{E},
\nonumber
\end{align}
yields
\beq
\label{mainFunct5}
\cFv_\eps[u]  :=
\frac{1}{\eps} \int_\Omega \left(W(u) - u^2 + (1-q) \eps^2 |\nabla u|^2 + u \left( {\bf 1} - \eps^2 \Delta \right)^{-1} u \right)\, dx. \nonumber
\eeq

\section*{Acknowledgements}
The authors warmly thank the Center for Nonlinear Analysis, where part of this research was carried out. 

\section*{Compliance with Ethical Standards}
Part of this research was carried out at the Center for Nonlinear Analysis. The center is partially supported by NSF Grant No.
DMS-0635983 and NSF PIRE Grant No. OISE-0967140.
The research of I. Fonseca
was partially funded by the National Science Foundation under Grant No.
DMS-0905778, DMS-1411646 and that of G. Leoni under Grant No. DMS-1007989, DMS-1412095. B. Zwicknagl acknowledges support by the Deutsche Forschungsgemeinschaft through the Sonderforschungsbereich 1060 {\em The mathematics of emergent effects}.

\nocite{kawakatsu1993phase}
\nocite{Seul:Science:1995}
\nocite{chermisi2010singular}
\nocite{baia2013coupled}

\bibliographystyle{plain}
\bibliography{andelman}

\end{document}